\documentclass[12pt,reqno]{amsart}%
\usepackage{amsmath, amsfonts, amssymb, amsthm, amscd, amsbsy}
\usepackage{fancyhdr}
\usepackage[usenames,dvipsnames,svgnames,x11names,hyperref]{xcolor}
\usepackage{geometry}
\usepackage{graphicx}
\usepackage{hyperref}
\usepackage{amsmath}
\usepackage{amsfonts}
\usepackage{amssymb}
\providecommand{\U}[1]{\protect\rule{.1in}{.1in}}
\hypersetup{
pagebackref=true,
hyperindex=true,
colorlinks=true,
breaklinks=true,
urlcolor=NavyBlue,
linkcolor=Fuchsia,
bookmarks=true,
bookmarksopen=false,
filecolor=black,
citecolor=ForestGreen,
linkbordercolor=red
}
\newtheorem{theorem}{Theorem}[section]
\theoremstyle{plain}

\newtheorem{corollary}{Corollary}[section]

\newtheorem{lemma}{Lemma}[section]

\newtheorem{proposition}{Proposition}[section]
\newtheorem{remark}{Remark}[section]

\numberwithin{equation}{section}
\allowdisplaybreaks
\AtBeginDocument{\hypersetup{pdfborder={0 0 0.01}}}
\newtheorem{theorema}{Theorem}[section]

\newtheorem{propositiona}{Proposition}[section]

\begin{document}
\title[Rellich-Sobolev interpolation inequality]{Sobolev interpolation inequalities with optimal Hardy-Rellich inequalities and
critical exponents}
\author{Nguyen Anh Dao}
\address{N. A. Dao: Institute of Applied Mathematics\\
University of Economics Ho Chi Minh City\\
Ho Chi Minh City, Viet Nam}
\email{anhdn@ueh.edu.vn}
\author{Anh Xuan Do}
\address{Anh Xuan Do: Department of Mathematics, University of Connecticut, Storrs, CT
06269, USA}
\email{anh.do@uconn.edu}
\author{Nguyen Lam}
\address{Nguyen Lam: School of Science and the Environment, Grenfell Campus, Memorial
University of Newfoundland, Corner Brook, NL A2H5G4, Canada}
\email{nlam@mun.ca}
\author{Guozhen Lu}
\address{Guozhen Lu: Department of Mathematics, University of Connecticut, Storrs, CT
06269, USA}
\email{guozhen.lu@uconn.edu}
\date{\today}
\thanks{N.A. Dao was partially supported by University of Economics Ho Chi Minh
	City (UEH), Vietnam. A. Do and G. Lu were partially supported by
	grants from the Simons Foundation and a Simons Fellowship. N. Lam
	was partially supported by an NSERC Discovery Grant.}

\begin{abstract}
We establish a new family of the critical higher order Sobolev interpolation
inequalities for radial functions as well as for non-radial functions. These
Sobolev interpolation inequalities are sharp in the sense that they use the
optimal quadratic forms of the sharp Hardy-Rellich inequalities and cover the
Sobolev critical exponents. Our results extend those studied by Dietze and Nam
in \cite{DN23} for the first order derivative case to higher order setting.
The well-known P\'{o}lya-Szeg\"{o} symmetrization principle and the nonlinear
ground state representation play an important role in the work of \cite{DN23}.
To overcome the absence of the P\'{o}lya-Szeg\"{o} principle and the nonlinear
ground state representation in the higher order case, our proofs rely on the
Fourier analysis and a higher order verion of the Talenti comparison
principle. We also study a new version of the critical Hardy-Sobolev
interpolation inequality involving the critical quadratic form of the Hardy
inequality and Lorentz norms. Our critical Hardy-Sobolev interpolation
inequality complements the result of Dietze and Nam in \cite{DN23}.

\end{abstract}
\maketitle

\section{Introduction}

In the two celebrated papers \cite{Gag59, Nir59}, Gagliardo and Nirenberg
proved the following Sobolev interpolation inequality, also known as the
Gagliardo-Nirenberg inequality:%

\[
\left\Vert u\right\Vert _{W^{k_{1},p_{1}}}^{\theta}\left\Vert u\right\Vert
_{W^{k_{2},p_{2}}}^{1-\theta}\gtrsim\left\Vert u\right\Vert _{W^{k,p}}\text{,
}\forall u\in W^{k_{1},p_{1}}\cap W^{k_{2},p_{2}}.
\]
Here $0\leq\theta\leq1$, $1\leq p$, $p_{1}$, $p_{2}\leq\infty$ and the
nonnegative integers $k$, $k_{1}$, $k_{2}$ satisfy the standard conditions of
the interpolation inequalities:
\begin{align*}
k  &  =\theta k_{1}+\left(  1-\theta\right)  k_{2}\\
\frac{1}{p}  &  =\frac{\theta}{p_{1}}+\frac{1-\theta}{p_{2}}.
\end{align*}
The Gagliardo-Nirenberg inequality plays vital roles in several areas of
mathematics such as in Sobolev spaces theory and in partial differential
equations. Therefore, in the past 80 years, it has been investigated
extensively and intensively in the literature. Given the enormous literature
on the interpolation inequalities, it is impossible to provide a comprehensive
list of bibliographies on the topic. Therefore, we just refer the interested
reader to \cite{BM18, BVY21, BV97, Dao23, DLL22, DLL22a, DN23, Lu96, Lu97,
Lu00, MNN21, PP14, VZ00}, to name just a few, for some variants and
generalizations on the Gagliardo-Nirenberg inequality that relate to this
paper's main purposes, and to the recent papers \cite{BCFGG20, CFLL24, CZ13, DY23,
DFLL23, DE20, DEL16, FFRS21, Fly20a, Fly20, MN22, VHN21}, for instance, for
vast references on this subject.

Gagliardo-Nirenberg inequality has also been investigated with Sobolev norms
replaced by other norms such as BMO norms, Morrey norms and Besov norms. For
instance, in \cite{PP14}, the authors studied the following
Gagliardo-Nirenberg inequality with Morrey norms: For $N\geq3$ and $1-\frac
{2}{N}\leq\theta\leq1$, we have
\begin{equation}
\left(  \int_{\mathbb{R}^{N}}|\nabla u|^{2}dx\right)  ^{\theta}\left(
\sup_{R>0,y\in\mathbb{R}^{N}}\frac{1}{R^{2}}\int_{B\left(  y,R\right)
}|u\left(  x\right)  |^{2}dx\right)  ^{1-\theta}\gtrsim\left\Vert u\right\Vert
_{L^{\frac{2N}{N-2}}}^{2},~\forall u\in D^{1,2}(\mathbb{R}^{N}). \label{SM}%
\end{equation}
Here $D^{k,2}(\mathbb{R}^{N})$ denotes the completion of $C_{0}^{\infty
}\left(  \mathbb{R}^{N}\right)  $ with respect to the semi-norm $\left(
\int_{\mathbb{R}^{N}}|\nabla^{k}u|^{2}dx\right)  ^{\frac{1}{2}}$. Obviously,
the above result implies that
\[
\left(  \int_{\mathbb{R}^{N}}|\nabla u|^{2}dx\right)  ^{\theta}\left(
\sup_{y\in\mathbb{R}^{N}}\int_{\mathbb{R}^{N}}\dfrac{|u(x)|^{2}}{|x-y|^{2}%
}dx\right)  ^{1-\theta}\gtrsim\left\Vert u\right\Vert _{L^{\frac{2N}{N-2}}%
}^{2},~\forall u\in D^{1,2}(\mathbb{R}^{N}).
\]
We also note here that the Morrey term in (\ref{SM}) can be improved to the
Besov norm. Indeed, it was showed in \cite{DLL22} by Dao, Lam and Lu that

\begin{theorema}
\label{Thm 1.2}Let $m$, $k$ be integers with $0\leq k<m$ and $s\geq0$ such
that $k+s>0$. Let $u\in\mathcal{S}^{\prime}(\mathbb{R}^{N})$ be such that
$\nabla^{m}u\in L^{p}(\mathbb{R}^{N})$, $1\leq p<\infty$ and $u\in\dot{B}%
^{-s}(\mathbb{R}^{N})$. Then, we have $\nabla^{k}u\in L^{r}(\mathbb{R}^{N})$,
$r=p\left(  \dfrac{m+s}{k+s}\right)  $, and
\begin{equation}
\left\Vert u\right\Vert _{\dot{B}^{-s}}^{\frac{m-k}{m+s}}\left\Vert \nabla
^{m}u\right\Vert _{L^{p}}^{\frac{k+s}{m+s}}\gtrsim\left\Vert \nabla
^{k}u\right\Vert _{L^{r}}, \label{SB}%
\end{equation}
where we denote the Besov space $\dot{B}^{-s}=\dot{B}_{\infty,\infty}^{-s}$.
\end{theorema}

Here, for any $s>0$, the homogeneous Besov space $\dot{B}_{\infty,\infty}%
^{-s}$ is the space of tempered distributions $u$ in $\mathcal{S}^{\prime
}\left(  \mathbb{R}^{N}\right)  $ verifying
\[
\left\Vert u\right\Vert _{\dot{B}^{-s}}:=\sup_{t>0}\left\{  t^{s/2}||K_{t}\ast
u||_{L^{\infty}}\right\}  <\infty,
\]
and $K_{t}(x)=\frac{e^{-\frac{\left\vert x\right\vert ^{2}}{4t}}}{(4\pi
t)^{N/2}}$ is the standard heat kernel.

Obviously, (\ref{SB}) is an improvement of (\ref{SM}) in the following sense.
First, we recall that a measurable function $u:\mathbb{R}^{N}\rightarrow
\mathbb{R}$ belongs to the Morrey space $M^{p,\alpha}(\mathbb{R}^{N})$ with
$p\in\lbrack1,\infty)$, and $\alpha\in\lbrack0,N]$, if and only if
\[
\left\Vert u\right\Vert _{M^{p,\alpha}}:=\sup_{R>0,x\in\mathbb{R}^{N}}\left(
R^{\alpha-N}\int_{B(x,R)}|u(y)|^{p}dy\right)  ^{1/p}<\infty.
\]

Then, we have the following embeddings (see, for instance \cite{PP14}):

\begin{proposition}
\label{lem 2.2} For any $\alpha\in(0, N)$ and $p>1$, we have

\begin{itemize}
\item[(i)] $M^{1, \alpha}(\mathbb{R}^{N})\hookrightarrow\dot{B}^{-\alpha
}(\mathbb{R}^{N});$

\item[(ii)] $M^{p,\alpha}(\mathbb{R}^{N})\hookrightarrow M^{1,\alpha
/p}(\mathbb{R}^{N}).$
\end{itemize}
\end{proposition}

\medskip To introduce the main goals of our paper, we now recall the classical
Hardy inequality: For $N\geq3$%
\[
h[u]:=\int_{\mathbb{R}^{N}}|\nabla u|^{2}dx-\dfrac{(N-2)^{2}}{4}\sup
_{y\in\mathbb{R}^{N}}\int_{\mathbb{R}^{N}}\dfrac{|u(x)|^{2}}{|x-y|^{2}}%
dx\geq0,\text{ }\forall u\in D^{1,2}(\mathbb{R}^{N}).
\]
The constant $\dfrac{(N-2)^{2}}{4}$ is optimal, but cannot be achieved by
nontrivial functions. It is also worth noting that the Hardy type
inequalities, together with the Sobolev type inequalities, are among the most
used inequalities in analysis and partial differential equations. There is
vast literature on the Hardy type inequalities and their applications.
Therefore, the interested reader is referred to the monographs \cite{BEL, GM1,
KP, OK}, to name just a few, that are by now standard references on Hardy type inequalities.

Since the operator $-\Delta-\dfrac{(N-2)^{2}}{4}\frac{1}{\left\vert
x\right\vert ^{2}}$ is nonnegative and is critical, it is also very
interesting and challenging to study the Sobolev interpolation inequalities
with the quadratic form $h[u]$. In this direction, Brezis and V\'{a}zquez in
\cite{BV97} and V\'{a}zquez and Zuazua in \cite{VZ00} proved that
\[
h[u]^{\theta}\left\Vert u\right\Vert _{2}^{2\left(  1-\theta\right)  }%
\gtrsim\left\Vert u\right\Vert _{q}^{2}%
\]
for $N\geq3$, $2<q<$ $\frac{2N}{N-2}$ and $\theta=N\left(  \frac{1}{2}%
-\frac{1}{q}\right)  $. We note that the condition $q<\frac{2N}{N-2}$ in the
above interpolation inequalities is necessary and cannot be improved to the
critical case $q=\frac{2N}{N-2}$. This result has also been extended by Frank
in \cite{Fra09} that%
\[
h[u]^{\theta}\left\Vert u\right\Vert _{2}^{2\left(  1-\theta\right)  }%
\gtrsim\left\Vert \left(  -\Delta\right)  ^{\frac{q}{2}}u\right\Vert _{2}%
^{2}.
\]
See also \cite{MNN21} for related results.

In an effort to investigate critical Sobolev interpolation inequalities with
the quadratic form $h[u]$, Dietze and Nam proved in a recent paper \cite{DN23}
the following Hardy-Sobolev interpolation inequalities:

\begin{theorema}
\label{Thm 1.1}(1) If $N\geq3$ and $\theta=\frac{1}{N}$, then
\begin{equation}
\left(  \int_{\mathbb{R}^{N}}|\nabla u|^{2}dx-\dfrac{(N-2)^{2}}{4}%
\int_{\mathbb{R}^{N}}\dfrac{|u(x)|^{2}}{|x|^{2}}dx\right)  ^{\theta}\left(
\int_{\mathbb{R}^{N}}\dfrac{|u(x)|^{2}}{|x|^{2}}dx\right)  ^{1-\theta}\geq
C\left\Vert u\right\Vert _{L^{\frac{2N}{N-2}}}^{2} \label{HSr}%
\end{equation}
holds with a constant $C=C(N,\theta)>0$ independent of $u\in D_{\text{rad}%
}^{1,2}(\mathbb{R}^{N})$. Moreover, \eqref{HSr} does not hold if $\theta
\neq1/N$.

(2) If $N=3$ and $\theta=\frac{1}{3}$, then the inequality
\begin{equation}
\left(  \int_{\mathbb{R}^{N}}|\nabla u|^{2}dx-\dfrac{(N-2)^{2}}{4}\sup
_{y\in\mathbb{R}^{N}}\int_{\mathbb{R}^{N}}\dfrac{|u(x)|^{2}}{|x-y|^{2}%
}dx\right)  ^{\theta}\left(  \sup_{y\in\mathbb{R}^{N}}\int_{\mathbb{R}^{N}%
}\dfrac{|u(x)|^{2}}{|x-y|^{2}}dx\right)  ^{1-\theta}\geq C\left\Vert
u\right\Vert _{L^{\frac{2N}{N-2}}}^{2} \label{HS}%
\end{equation}
holds with a constant $C=C(N,\theta)>0$ independent of $u\in D^{1,2}%
(\mathbb{R}^{N})$. Furthermore, \eqref{HS} does not hold if $N\geq4$ or
$\theta\neq1/3$.
\end{theorema}

Here $D_{\text{rad}}^{k,2}(\mathbb{R}^{N}):=\left\{  u\in D^{k,2}%
(\mathbb{R}^{N}):u\text{ is radially symmetric}\right\}  $.

The significance of the above results of Dietze and Nam is that the first
terms of \eqref{HSr} and \eqref{HS} are the quadratic form of the Hardy
inequality $h[u]$, and that we can improve the RHS of the Sobolev
interpolations to the critical Sobolev norm $\left\Vert u\right\Vert
_{L^{\frac{2N}{N-2}}}$. However, we note that \eqref{HS} holds for the special
case $N=\frac{1}{\theta}=3$ only. Also, \eqref{HSr} only holds for radial
functions and with $\theta=\frac{1}{N}$.

\medskip

The principal goals of our paper are to extend the results of Dietze and Nam
to the higher order cases. To serve our purposes, we now recall the higher
order Hardy inequalities, namely, the Hardy-Rellich inequalities:

\begin{theorema}
\label{Thm 1.3}Let $0<k<\frac{N}{2}$ be an integer. We have for all $u\in
D^{k,2}(\mathbb{R}^{N})$,
\begin{align*}
\int_{\mathbb{R}^{N}}|\Delta^{m}u|^{2}dx  &  \geq\dfrac{N^{2}}{4}%
\int_{\mathbb{R}^{N}}\dfrac{|\nabla\Delta^{m-1}u|^{2}}{|x|^{2}}dx\geq
\frac{4C_{1}(N,m)}{\left(  N-4m\right)  ^{2}}\int_{\mathbb{R}^{N}}%
\dfrac{|\nabla u|^{2}}{|x|^{4m-2}}dx\\
&  \geq C_{1}(N,m)\int_{\mathbb{R}^{N}}\dfrac{|u|^{2}}{|x|^{4m}}dx
\end{align*}
if $k=2m$, and
\begin{align*}
\int_{\mathbb{R}^{N}}|\nabla\Delta^{m}u|^{2}dx  &  \geq\dfrac{\left(
N-2\right)  ^{2}}{4}\int_{\mathbb{R}^{N}}\frac{|\Delta^{m}u|^{2}}{|x|^{2}%
}dx\geq\frac{4C_{2}(N,m)}{\left(  N-4m\right)  ^{2}}\int_{\mathbb{R}^{N}%
}\dfrac{|\nabla u|^{2}}{|x|^{4m-2}}dx\\
&  \geq C_{2}(N,m)\int_{\mathbb{R}^{N}}\dfrac{|u|^{2}}{|x|^{4m+2}}dx
\end{align*}
if $k=2m+1$. Here the constants%
\[
C_{1}(N,m)=\left(  \prod_{i=0}^{m-1}\dfrac{(N+4i)(N-4-4i)}{4}\right)  ^{2}%
\]
and
\[
C_{2}(N,m)=\left(  \dfrac{N-2}{2}\right)  ^{2}\left(  \prod_{i=0}^{m-1}%
\dfrac{(N+2+4i)(N-6-4i)}{4}\right)  ^{2}%
\]
are sharp.
\end{theorema}

These Hardy-Rellich type inequalities can be found in \cite{DH98, GLMW18,
Her77, Yav99}, for instance.

\medskip

Motivated by the results of Dietze and Nam in \cite{DN23}, the first main goal
of this paper is to establish the following critical Rellich-Sobolev
interpolation inequality with the critical quadratic forms of the Rellich
inequalities for radial functions:

\begin{theorem}
\label{T3}Let $m\in\mathbb{N}$, $N>4m$, $\frac{1}{N}\leq\theta\leq\frac{2m}%
{N}$ and $C_{1}(N,m)$ be the optimal constant defined in Theorem
\ref{Thm 1.3}, we have
\begin{equation}
\left(  \int_{\mathbb{R}^{N}}|\Delta^{m}u|^{2}dx-C_{1}(N,m)\int_{\mathbb{R}%
^{N}}\dfrac{|u|^{2}}{|x|^{4m}}dx\right)  ^{\theta}\left(  \int_{\mathbb{R}%
^{N}}\dfrac{|u|^{2}}{|x|^{4m}}dx\right)  ^{1-\theta}\geq C\left\Vert
u\right\Vert _{L^{\frac{2N}{N-4m}}}^{2} \label{2m-1i}%
\end{equation}
holds with a constant $C=C(N,m,\theta)>0$ independent of $u\in D_{\text{rad}%
}^{2m,2}(\mathbb{R}^{N})$.

Similarly, with $N>4m+2$, $\frac{1}{N}\leq\theta\leq\frac{2m+1}{N}$ and
$C_{2}(N,m)$ be the sharp constant defined in Theorem \ref{Thm 1.3}, we have
\begin{equation}
\left(  \int_{\mathbb{R}^{N}}|\nabla\Delta^{m}u|^{2}dx-C_{2}(N,m)\int
_{\mathbb{R}^{N}}\dfrac{|u|^{2}}{|x|^{4m+2}}dx\right)  ^{\theta}\left(
\int_{\mathbb{R}^{N}}\dfrac{|u|^{2}}{|x|^{4m+2}}dx\right)  ^{1-\theta}\geq
C\left\Vert u\right\Vert _{L^{\frac{2N}{N-4m-2}}}^{2} \label{odd-1i}%
\end{equation}
holds with a constant $C=C(N,m,\theta)>0$ independent of $u\in D_{\text{rad}%
}^{2m+1,2}(\mathbb{R}^{N})$.
\end{theorem}

The inequality (\ref{2m-1i}) implies that
\[
\left(  \int_{\mathbb{R}^{N}}|\Delta^{m}u|^{2}dx\right)  ^{\theta}\left(
\int_{\mathbb{R}^{N}}\dfrac{|u|^{2}}{|x|^{4m}}dx\right)  ^{1-\theta}%
\gtrsim\left\Vert u\right\Vert _{L^{\frac{2N}{N-4m}}}^{2}.
\]
Since $u$ is radial, $\int_{\mathbb{R}^{N}}|\Delta^{m}u|^{2}dx\lesssim\int
_{0}^{\infty}\left\vert u^{\left(  2m\right)  }\right\vert ^{2}r^{N-1}dr$.
This implies that
\[
\left(  \int_{0}^{\infty}\left\vert u^{\left(  2m\right)  }\right\vert
^{2}r^{N-1}dr\right)  ^{\theta}\left(  \int_{0}^{\infty}|u|^{2}r^{N-1-4m}%
dr\right)  ^{1-\theta}\gtrsim\left(  \int_{0}^{\infty}|u|^{\frac{2N}{N-4m}%
}r^{N-1}dr\right)  ^{\frac{N-4m}{N}}.
\]
This is a special case of the following higher order Caffarelli-Kohn-Nirenberg
inequality (see \cite{Lin86}, for instance):

\begin{propositiona}
[Higher order Caffarelli-Kohn-Nirenberg inequality]\label{HCKN}There exists a
positive constant $C$ such that the following inequality holds true for all
$u\in C_{0}^{\infty}\left(  \mathbb{R}^{n}\right)  :$%
\[
\left\Vert \left\vert x\right\vert ^{\gamma}D^{j}u\right\Vert _{L^{s}}\leq
C\left\Vert \left\vert x\right\vert ^{\alpha}D^{k}u\right\Vert _{L^{p}}%
^{a}\left\Vert \left\vert x\right\vert ^{\beta}u\right\Vert _{L^{q}}^{1-a}%
\]
if and only if the following conditions hold:%
\begin{align}
\frac{1}{s}+\frac{\gamma-j}{n}  &  =a\left(  \frac{1}{p}+\frac{\alpha-k}%
{n}\right)  +\left(  1-a\right)  \left(  \frac{1}{q}+\frac{\beta}{n}\right)
\label{DB}\\
\gamma &  \leq a\alpha+\left(  1-a\right)  \beta\label{EC}\\
a\left(  \alpha-k\right)  +\left(  1-a\right)  \beta+j  &  \leq\gamma\text{ if
}s>0\text{ and }\frac{1}{q}+\frac{\beta}{n}=\frac{1}{p}+\frac{\alpha-k}{n}.
\end{align}

\end{propositiona}

In particular, the condition (\ref{EC}) is equivalent to $\theta\geq\frac
{1}{N}$. Therefore, the condition $\theta\geq\frac{1}{N}$ is necessary in
(\ref{2m-1i}). Similarly, $\theta\geq\frac{1}{N}$ is also necessary in
(\ref{odd-1i}).

It is worth mentioning that the proof of (\ref{HSr}) in Theorem \ref{Thm 1.1}
in \cite{DN23} made use of the nonlinear ground state representations for the
classical Hardy inequalities (see \cite{FLS08, FS08}, for instance), and then
by making change of variables in a clever way, Dietze and Nam were able to
reduce the question to a $1$-dimensional Caffarelli-Kohn-Nirenberg inequality
that has been investigated in, for instance, \cite{N41}. However, a nonlinear
ground state representation for higher order Hardy type inequalities, namely,
Rellich type inequalities, is still absent in the literature. In order to
overcome the obstacles and to establish our Theorem \ref{T3}, we will first
use Fourier analysis and prove the following two results, that seem new in the
literature and are of independent interest, and then use them to deduce our
main result Theorem \ref{T3}.

Let $\mathcal{F}_{k}\left(  \phi\right)  \left(  \xi\right)  =\int
_{\mathbb{R}^{k}}\phi\left(  x\right)  e^{-2\pi ix\cdot\xi}dx$ be the Fourier
transform of $\phi$ on $\mathbb{R}^{k}$. Then, we have

\begin{theorem}
\label{T3.1}Let $u\left(  \cdot\right)  =U\left(  \left\vert \cdot\right\vert
\right)  $ and $v\left(  \cdot\right)  =V\left(  \left\vert \cdot\right\vert
\right)  $ be such that $\mathcal{F}_{N}\left(  u\right)  $ and $\mathcal{F}%
_{N+2}\left(  v\right)  $ are well-defined, $\lim_{r\rightarrow0^{+}}U\left(
r\right)  r^{N}=0$, $\lim_{r\rightarrow\infty}U\left(  r\right)  r^{\frac
{N-1}{2}}=0$, and $rV\left(  r\right)  =U^{\prime}\left(  r\right)  $ for
$r>0$. Then $\mathcal{F}_{N+2}\left(  v\right)  =-2\pi\mathcal{F}_{N}\left(
u\right)  $. As a consequence, for $\alpha\geq0$, if defined, then
\[
\int_{\mathbb{R}^{N}}|\Delta^{\alpha+\frac{1}{2}}u|^{2}dx=\frac{|\mathbb{S}%
^{N-1}|}{|\mathbb{S}^{N+1}|}\int_{\mathbb{R}^{N+2}}|\Delta^{\alpha}v|^{2}dx.
\]

\end{theorem}

\begin{theorem}
\label{T3.2}Let $U$ be a smooth function on $\left(  0,\infty\right)  $ and
$u\left(  \cdot\right)  =U\left(  \left\vert \cdot\right\vert \right)  $ be a
radial function. If $N\geq2m-1$, $m\in\mathbb{N}$, and $\lim_{r\rightarrow
0^{+},\infty}\left(  \frac{U^{\prime}\left(  r\right)  }{r}\right)  ^{\left(
2m-2\right)  }r^{\frac{N}{2}}=0$, then
\[
\int_{\mathbb{R}^{N}}|\Delta^{m}u|^{2}dx\geq|\mathbb{S}^{N-1}|\int_{0}%
^{\infty}\left\vert U^{\left(  2m\right)  }\left(  r\right)  \right\vert
^{2}r^{N-1}dr.
\]
If $N\geq2m$, $m\in\mathbb{N}$, and $\lim_{r\rightarrow0^{+},\infty}\left(
\frac{U^{\prime}\left(  r\right)  }{r}\right)  ^{\left(  2m-1\right)
}r^{\frac{N}{2}}=0$, then
\[
\int_{\mathbb{R}^{N}}|\nabla\Delta^{m}u|^{2}dx\geq|\mathbb{S}^{N-1}|\int
_{0}^{\infty}\left\vert U^{\left(  2m+1\right)  }\left(  r\right)  \right\vert
^{2}r^{N-1}dr.
\]

\end{theorem}

\medskip

If one wishes to extend the results in Theorem \ref{T3} to the setting of
non-radial functions, one will need to put more constraints on the dimensions,
as showed in \cite{DN23}. Indeed, the second purpose of this article is to
study the following critical Rellich-Sobolev interpolation inequality with the
critical quadratic forms of the Rellich inequalities for general functions:

\begin{theorem}
\label{T1} With $N=1/\theta=4m+1$ $(m\in\mathbb{Z}_{+})$ and $C_{1}(N,m)$
defined in Theorem \ref{Thm 1.3}, we have
\begin{align}
\left(  \int_{\mathbb{R}^{N}}|\Delta^{m}u|^{2}dx-C_{1}(N,m)\sup_{y\in
\mathbb{R}^{N}}\int_{\mathbb{R}^{N}}\dfrac{|u(x)|^{2}}{|x-y|^{4m}}dx\right)
^{\theta}  &  \left(  \sup_{y\in\mathbb{R}^{N}}\int_{\mathbb{R}^{N}}%
\dfrac{|u(x)|^{2}}{|x-y|^{4m}}dx\right)  ^{1-\theta}\nonumber\\
&  \geq C\left\Vert u\right\Vert _{L^{\frac{2N}{N-4m}}}^{2} \label{2m-0i}%
\end{align}
holds with a constant $C=C(N,m,\theta)>0$ independent of $u\in D^{2m,2}%
(\mathbb{R}^{N})$.

Also, with $N=1/\theta=4m+3$ $(m\in\mathbb{N})$ and $C_{2}(N,m)$ defined in
Theorem \ref{Thm 1.3}, we have
\begin{align}
\left(  \int_{\mathbb{R}^{N}}|\nabla\Delta^{m}u|^{2}dx-C_{2}(N,m)\sup
_{y\in\mathbb{R}^{N}}\int_{\mathbb{R}^{N}}\dfrac{|u(x)|^{2}}{|x-y|^{4m+2}%
}dx\right)  ^{\theta}  &  \left(  \sup_{y\in\mathbb{R}^{N}}\int_{\mathbb{R}%
^{N}}\dfrac{|u(x)|^{2}}{|x-y|^{4m+2}}dx\right)  ^{1-\theta}\nonumber\\
&  \geq C\left\Vert u\right\Vert _{L^{\frac{2N}{N-4m-2}}}^{2} \label{odd-0i}%
\end{align}
holds with a constant $C=C(N,m,\theta)>0$ independent of $u\in D^{2m+1,2}%
(\mathbb{R}^{N})$.
\end{theorem}

In particular, by choosing $m=0$ in (\ref{odd-0i}), we recover the result
studied by Dietze and Nam in \cite{DN23}.

\begin{remark}
As in \cite{DN23}, the reason one needs to include the supremums in
(\ref{2m-0i}) and (\ref{odd-0i}) is to make the inequalities (\ref{2m-0i}) and
(\ref{odd-0i}) to be invariant under translations and dilations. Moreover,
they cannot be removed since without supremum, the term $\int_{\mathbb{R}^{N}%
}\dfrac{|u(x)|^{2}}{|x|^{k}}dx$ is not translation-invariant and can be made
arbitrarily small by the translation $u_{z}\left(  \cdot\right)  =u\left(
x-z\right)  $ as $\left\vert z\right\vert \rightarrow\infty$. Therefore, by
the critical Sobolev embeddings, (\ref{2m-0i}) and (\ref{odd-0i}) cannot hold
without the supremum terms.
\end{remark}

It is also worthy to mention that the proof of (\ref{HS}) in Theorem
\ref{Thm 1.1} in \cite{DN23} is splitted into two cases. In the first case
when the Hardy quadratic form $h[u]$ is comparable to the standard Dirichlet
energy $\frac{1}{2}\left\Vert \nabla u\right\Vert _{2}^{2}$, the authors can
apply the Gagliardo-Nirenberg inequality with Morrey norms established in
\cite{PP14}. In the second case, the proof relies on the P\'{o}lya-Szeg\"{o}
rearrangement principle to reduce the problem to the radial case. In our
situation, to deal with the first case, we can use the higher order
Gagliardo-Nirenberg inequality with Besov norms established recently by Dao,
Lam and Lu in \cite{DLL22} (We note that our results in \cite{Dao23, DLL22}
are stronger than the ones established in \cite{PP14} in the sense that the
Morrey spaces can be embedded into the Besov spaces. See Proposition
\ref{lem 2.2}). However, in the second case, a P\'{o}lya-Szeg\"{o} type
inequality for higher order operators is completely missing in the literature
and is not expected to hold. To by-pass the difficulties, we will apply the
comparison principle for higher order derivatives which is a generalization of
the comparison principle in \cite{Tal76}.

\medskip

We also mention here that Dietze and Nam proved in \cite{DN23} the following
Hardy-Sobolev inequalities with Lorentz norms:

\begin{theorema}
\label{Thm 1.4} Let $2\leq p<N$, $p\leq r\leq\infty$, $p^{\ast}=\frac{Np}%
{N-p}$ and
\[
\theta\in\left[  \frac{p}{\min\left(  r,p^{\ast}\right)  },\frac{1}{p}%
-\frac{1}{r}\right]  .
\]
Then
\begin{equation}
\left(  \int_{\mathbb{R}^{N}}|\nabla u|^{p}dx-\left(  \dfrac{N-p}{p}\right)
^{p}\sup_{y\in\mathbb{R}^{N}}\int_{\mathbb{R}^{N}}\dfrac{|u(x)|^{p}}%
{|x-y|^{p}}dx\right)  ^{\theta}\left(  \sup_{y\in\mathbb{R}^{N}}%
\int_{\mathbb{R}^{N}}\dfrac{|u(x)|^{p}}{|x-y|^{p}}dx\right)  ^{1-\theta}\geq
C\left\Vert u\right\Vert _{L^{p^{\ast},r}}^{p} \label{HSL}%
\end{equation}
holds with a constant $C=C(N,p,r,\theta)>0$ independent of $u\in
D^{1,p}(\mathbb{R}^{N})$. Moreover, \eqref{HSL} does not hold if $\theta
<\frac{p}{\min\left(  r,\frac{pN}{N-p}\right)  }$ (with arbitrary $p\geq2$) or
if $\theta>\frac{1}{p}-\frac{1}{r}$ and $p=2.$
\end{theorema}

Our last purpose in this paper is to complement the above result of Dietze and
Nam in \cite{DN23} by proving a new version of the critical Hardy-Sobolev
interpolation inequality with the critical quadratic forms of the Hardy
inequality and Lorentz norms. More precisely, our next result can be read as follows:

\begin{theorem}
\label{T2}Let $2\leq q\leq p<N$, $r=\frac{Nq}{N-p}$ be such that $\max\left\{
\dfrac{N}{N-p},p\right\}  \leq r\leq\min\left\{  \dfrac{Np}{(N-p)^{2}}%
,\dfrac{Np}{N-p}\right\}  $ and $\dfrac{N-p}{N}\leq\theta\leq\dfrac{p}{Nq}$.
For all $u\in C_{0}^{\infty}(\mathbb{R}^{N})$, we have
\begin{align}
\left(  \left\Vert \nabla u\right\Vert _{p,q}^{q}-\omega_{N}^{\frac{q}{p}%
-1}\left(  \dfrac{N-p}{p}\right)  ^{q}\sup_{y\in\mathbb{R}^{N}}\int
_{\mathbb{R}^{N}}\dfrac{|u(x)|^{q}}{|x-y|^{N+q-\frac{Nq}{p}}}dx\right)
^{\theta}  &  \left(  \sup_{y\in\mathbb{R}^{N}}\int_{\mathbb{R}^{N}}%
\dfrac{|u(x)|^{q}}{|x-y|^{N+q-\frac{Nq}{p}}}dx\right)  ^{1-\theta}\nonumber\\
&  \geq C\left\Vert u\right\Vert _{L^{p^{\ast},r}}^{q}, \label{eq 3.4}%
\end{align}
with a constant $C=C(N,p,q,r,\theta)>0$ independent of $u$.
\end{theorem}

We note that the first factor on the left hand side of the above inequality is
nonnegative thanks to the following Hardy inequality with Lorentz norm:

\begin{propositiona}
[Theorem $1.2$, \cite{D22}]Let $1\leq q\leq p<N$. For $u\in C_{0}^{\infty
}(\mathbb{R}^{N})$, there holds
\[
\left\Vert \nabla u\right\Vert _{p,q}^{q}\geq\omega_{N}^{\frac{q}{p}-1}\left(
\dfrac{N-p}{p}\right)  ^{q}\int_{\mathbb{R}^{N}}\dfrac{|u|^{q}}{|x|^{N+q-\frac
{Nq}{p}}}dx.
\]
The constant $\omega_{N}^{\frac{q}{p}-1}\left(  \frac{N-p}{p}\right)  ^{q}$ is
sharp and cannot be attained by nontrivial functions in the Sobolev-Lorentz
space $W^{1}L^{p,q}(\mathbb{R}^{N})$, which is the closure of smooth compactly
supported functions with respect to the norm $\left\Vert \nabla u\right\Vert
_{p,q}+\left\Vert u\right\Vert _{p,q}$.
\end{propositiona}

Also, when $p=q$ in Theorem \ref{T2}, we deduce that

\begin{corollary}
Let $\max\left\{  2,N-1\right\}  \leq p<N$. Then%
\begin{align*}
&  \left(  \left\Vert \nabla u\right\Vert _{p}^{p}-\left(  \dfrac{N-p}%
{p}\right)  ^{p}\sup_{y\in\mathbb{R}^{N}}\int_{\mathbb{R}^{N}}\dfrac
{|u(x)|^{p}}{|x-y|^{p}}dx\right)  ^{\theta}\left(  \sup_{y\in\mathbb{R}^{N}%
}\int_{\mathbb{R}^{N}}\dfrac{|u(x)|^{p}}{|x-y|^{p}}dx\right)  ^{1-\theta}\\
&  \geq C\left\Vert u\right\Vert _{L^{p^{\ast}}}^{p}%
\end{align*}
with
\[
\dfrac{N-p}{N}\leq\theta\leq\dfrac{1}{N}.
\]

\end{corollary}

\medskip

Our paper is organized as follows: In Section 2, we will use Fourier
techniques to prove Theorem \ref{T3.1} and Theorem \ref{T3.2}, which are of
independent interest, and then use them to prove our first main result of the
paper, namely, Theorem \ref{T3}. In Section 3, we will prove the second
principal result of our paper, that is Theorem \ref{T1}. Finally, the critical
Hardy-Sobolev interpolation inequality with the critical quadratic forms of
the Hardy inequality and Lorentz norms Theorem \ref{T2} will be established in
Section 4.

\section{Proofs of Theorem \ref{T3.1}, Theorem \ref{T3.2} and Theorem
\ref{T3}}

\begin{proof}
[Proof of Theorem \ref{T3.1}]Since $v$ is radial, we have
\begin{align*}
\mathcal{F}_{N+2}\left(  v\right)  \left(  r\right)   &  =2\pi r^{-\frac{N}%
{2}}\int_{0}^{\infty}V\left(  s\right)  J_{\frac{N}{2}}\left(  2\pi rs\right)
s^{\frac{N+2}{2}}ds\\
&  =2\pi r^{-\frac{N}{2}}\int_{0}^{\infty}U^{\prime}\left(  s\right)
J_{\frac{N}{2}}\left(  2\pi rs\right)  s^{\frac{N}{2}}ds\\
&  =-\left(  2\pi\right)  ^{2-\frac{N+2}{2}}r^{-N}\int_{0}^{\infty}U\left(
s\right)  \frac{d\left(  J_{\frac{N}{2}}\left(  2\pi rs\right)  \left(  2\pi
rs\right)  ^{\frac{N}{2}}\right)  }{ds}ds.
\end{align*}
Here $J_{\nu}$ is the classical Bessel function of order $\nu$. See, for
instance, \cite{SW71}. Note that we can apply the integration by parts here
because of the assumptions on $U$ and the fact that $\left\vert J_{\frac{N}%
{2}}\left(  s\right)  \right\vert \lesssim\frac{s^{\frac{N}{2}}}{\left(
1+s\right)  ^{\frac{N}{2}+\frac{1}{2}}}$, $s>0$. Now, by applying the identity
(see \cite{Watson})
\[
\frac{d}{dx}\left[  x^{\nu}J_{\nu}\left(  x\right)  \right]  =x^{\nu}J_{\nu
-1}\left(  x\right)  ,
\]
we obtain%
\begin{align*}
\mathcal{F}_{N+2}\left(  v\right)  \left(  r\right)   &  =-\left(
2\pi\right)  ^{2-\frac{N+2}{2}}r^{-N}\int_{0}^{\infty}U\left(  s\right)
J_{\frac{N-2}{2}}\left(  2\pi rs\right)  \left(  2\pi rs\right)  ^{\frac{N}%
{2}}2\pi rds\\
&  =-\left(  2\pi\right)  ^{2}r^{-\frac{N-2}{2}}\int_{0}^{\infty}U\left(
s\right)  J_{\frac{N-2}{2}}\left(  2\pi rs\right)  s^{\frac{N}{2}}ds\\
&  =-2\pi\mathcal{F}_{N}\left(  u\right)  .
\end{align*}
Therefore
\begin{align*}
\int_{\mathbb{R}^{N}}|\Delta^{\alpha+\frac{1}{2}}u|^{2}dx  &  =\left(
2\pi\right)  ^{4\alpha+2}\int_{\mathbb{R}^{N}}\left\vert \xi\right\vert
^{4\alpha+2}\left\vert \mathcal{F}_{N}\left(  u\right)  \left(  \xi\right)
\right\vert ^{2}d\xi\\
&  =\left(  2\pi\right)  ^{4\alpha+2}|\mathbb{S}^{N-1}|\int_{0}^{\infty
}r^{N-1+4\alpha+2}\left\vert \mathcal{F}_{N}\left(  u\right)  \left(
r\right)  \right\vert ^{2}dr\\
&  =\left(  2\pi\right)  ^{4\alpha}|\mathbb{S}^{N-1}|\int_{0}^{\infty
}r^{N+1+4\alpha}\left\vert \mathcal{F}_{N+2}\left(  v\right)  \left(
r\right)  \right\vert ^{2}dr\\
&  =\frac{|\mathbb{S}^{N-1}|}{|\mathbb{S}^{N+1}|}\int_{\mathbb{R}^{N+2}%
}|\Delta^{\alpha}v|^{2}dx.
\end{align*}

\end{proof}

\begin{proof}
[Proof of Theorem \ref{T3.2}]We will prove by induction. Indeed, for the base
case, it is easy to see that
\[
\int_{\mathbb{R}^{N}}|\nabla u|^{2}dx\geq|\mathbb{S}^{N-1}|\int_{0}^{\infty
}\left\vert U^{\prime}\left(  r\right)  \right\vert ^{2}r^{N-1}dr
\]
for all $N\geq1$. Also, for all $N\geq1$,
\begin{align*}
\int_{\mathbb{R}^{N}}|\Delta u|^{2}dx  &  =|\mathbb{S}^{N-1}|\int_{0}^{\infty
}\left\vert U^{^{\prime\prime}}\left(  r\right)  +\frac{N-1}{r}U^{\prime
}\left(  r\right)  \right\vert ^{2}r^{N-1}dr\\
&  =|\mathbb{S}^{N-1}|\int_{0}^{\infty}\left\vert U^{^{\prime\prime}}\left(
r\right)  \right\vert ^{2}r^{N-1}dr+|\mathbb{S}^{N-1}|\left(  N-1\right)
\int_{0}^{\infty}\left\vert U^{\prime}\left(  r\right)  \right\vert
^{2}r^{N-3}dr\\
&  \geq|\mathbb{S}^{N-1}|\int_{0}^{\infty}\left\vert U^{^{\prime\prime}%
}\left(  r\right)  \right\vert ^{2}r^{N-1}dr.
\end{align*}
Now, assume that
\[
\int_{\mathbb{R}^{N}}|\nabla\Delta^{k-1}u|^{2}dx\geq|\mathbb{S}^{N-1}|\int
_{0}^{\infty}\left\vert U^{\left(  2k-1\right)  }\left(  r\right)  \right\vert
^{2}r^{N-1}dr
\]
for $N\geq2k-2$ and
\[
\int_{\mathbb{R}^{N}}|\Delta^{k}u|^{2}dx\geq|\mathbb{S}^{N-1}|\int_{0}%
^{\infty}\left\vert U^{\left(  2k\right)  }\left(  r\right)  \right\vert
^{2}r^{N-1}dr
\]
for $N\geq2k-1$. Then we will first prove that
\[
\int_{\mathbb{R}^{N}}|\nabla\Delta^{k}u|^{2}dx\geq|\mathbb{S}^{N-1}|\int
_{0}^{\infty}\left\vert U^{\left(  2k+1\right)  }\left(  r\right)  \right\vert
^{2}r^{N-1}dr
\]
for $N\geq2k$. Indeed, let $V=\frac{U^{\prime}}{r}$. Then, by Theorem
\ref{T3.1}, we get%
\begin{align*}
\int_{\mathbb{R}^{N}}|\nabla\Delta^{k}u|^{2}dx  &  =\left(  2\pi\right)
^{4k+2}\int_{\mathbb{R}^{N}}\left\vert \xi\right\vert ^{4k+2}\left\vert
\mathcal{F}_{N}\left(  u\right)  \left(  \xi\right)  \right\vert ^{2}d\xi\\
&  =\left(  2\pi\right)  ^{4k+2}|\mathbb{S}^{N-1}|\int_{0}^{\infty
}r^{N-1+4k+2}\left\vert \mathcal{F}_{N}\left(  u\right)  \left(  r\right)
\right\vert ^{2}dr\\
&  =\left(  2\pi\right)  ^{4k}|\mathbb{S}^{N-1}|\int_{0}^{\infty}%
r^{N-1+4k+2}\left\vert \mathcal{F}_{N+2}\left(  v\right)  \left(  r\right)
\right\vert ^{2}dr\\
&  =\frac{|\mathbb{S}^{N-1}|}{|\mathbb{S}^{N+1}|}\int_{\mathbb{R}^{N+2}%
}|\Delta^{k}v|^{2}dx.
\end{align*}
Therefore,
\[
\int_{\mathbb{R}^{N}}|\nabla\Delta^{k}u|^{2}dx=\frac{|\mathbb{S}^{N-1}%
|}{|\mathbb{S}^{N+1}|}\int_{\mathbb{R}^{N+2}}|\Delta^{k}v|^{2}dx\geq
|\mathbb{S}^{N-1}|\int_{0}^{\infty}\left\vert V^{\left(  2k\right)  }\left(
r\right)  \right\vert ^{2}r^{N+1}dr.
\]
Note that since $U^{\prime}=rV$, we have $U^{\left(  2k+1\right)  }\left(
r\right)  =rV^{\left(  2k\right)  }\left(  r\right)  +2kV^{\left(
2k-1\right)  }$. Therefore
\begin{align*}
\int_{0}^{\infty}\left\vert U^{\left(  2k+1\right)  }\left(  r\right)
\right\vert ^{2}r^{N-1}dr  &  =\int_{0}^{\infty}\left\vert V^{\left(
2k\right)  }\left(  r\right)  \right\vert ^{2}r^{N+1}dr+2k\left(  2k-N\right)
\int_{0}^{\infty}\left\vert V^{\left(  2k-1\right)  }\left(  r\right)
\right\vert ^{2}r^{N-1}dr\\
&  \leq\int_{0}^{\infty}\left\vert V^{\left(  2k\right)  }\left(  r\right)
\right\vert ^{2}r^{N+1}dr.
\end{align*}
Next, we will show that
\[
\int_{\mathbb{R}^{N}}|\Delta^{k+1}u|^{2}dx\geq|\mathbb{S}^{N-1}|\int
_{0}^{\infty}\left\vert U^{\left(  2k+2\right)  }\left(  r\right)  \right\vert
^{2}r^{N-1}dr
\]
for $N\geq2k+1$. Indeed, again, we let $V=\frac{U^{\prime}}{r}$. Then, by
Theorem \ref{T3.1}, we obtain%
\begin{align*}
\int_{\mathbb{R}^{N}}|\Delta^{k+1}u|^{2}dx  &  =\left(  2\pi\right)
^{4k+4}|\mathbb{S}^{N-1}|\int_{0}^{\infty}r^{N-1+4k+4}\left\vert
\mathcal{F}_{N}\left(  u\right)  \left(  r\right)  \right\vert ^{2}dr\\
&  =\left(  2\pi\right)  ^{4k+2}|\mathbb{S}^{N-1}|\int_{0}^{\infty
}r^{N-1+4k+4}\left\vert \mathcal{F}_{N+2}\left(  v\right)  \left(  r\right)
\right\vert ^{2}dr\\
&  =\frac{|\mathbb{S}^{N-1}|}{|\mathbb{S}^{N+1}|}\int_{\mathbb{R}^{N+2}%
}|\nabla\Delta^{k}v|^{2}dx.
\end{align*}
Hence%
\[
\int_{\mathbb{R}^{N}}|\Delta^{k+1}u|^{2}dx=\frac{|\mathbb{S}^{N-1}%
|}{|\mathbb{S}^{N+1}|}\int_{\mathbb{R}^{N+2}}|\nabla\Delta^{k}v|^{2}%
dx\geq|\mathbb{S}^{N-1}|\int_{0}^{\infty}\left\vert V^{\left(  2k+1\right)
}\left(  r\right)  \right\vert ^{2}r^{N+1}dr.
\]
Similarly, we have $U^{\left(  2k+2\right)  }\left(  r\right)  =rV^{\left(
2k+1\right)  }\left(  r\right)  +\left(  2k+1\right)  V^{\left(  2k\right)  }$
and
\begin{align*}
\int_{0}^{\infty}\left\vert U^{\left(  2k+2\right)  }\left(  r\right)
\right\vert ^{2}r^{N-1}dr  &  =\int_{0}^{\infty}\left\vert V^{\left(
2k+1\right)  }\left(  r\right)  \right\vert ^{2}r^{N+1}dr+\left(  2k+1\right)
\left(  2k+1-N\right)  \int_{0}^{\infty}\left\vert V^{\left(  2k\right)
}\left(  r\right)  \right\vert ^{2}r^{N-1}dr\\
&  \leq\int_{0}^{\infty}\left\vert V^{\left(  2k+1\right)  }\left(  r\right)
\right\vert ^{2}r^{N+1}dr.
\end{align*}

\end{proof}

\begin{proof}
[Proof of Theorem \ref{T3}]We will first prove (\ref{2m-1i}). Indeed, let
\[
A=\int_{\mathbb{R}^{N}}|\Delta^{m}u|^{2}dx,\;\;B=C_{1}(N,m)\int_{\mathbb{R}%
^{N}}\dfrac{|u|^{2}}{|x|^{4m}}dx.
\]
\textbf{Case 1:} $\theta=\frac{1}{N}$.

If $(1-\theta)A<B\leq A$, then it is easy to check that
\[
\dfrac{d}{dB}((A-B)^{\theta}B^{1-\theta})=((1-\theta)A-B)(A-B)^{1-\theta
}B^{-\theta}.
\]
Hence, the function $(A-B)^{\theta}B^{1-\theta}$ is non-increasing with
respect to $B$ if $0\leq(1-\theta)A<B\leq A$. Therefore, by the Hardy-Rellich
inequalities (Theorem \ref{Thm 1.3}),
\begin{align*}
B=C_{1}(N,m)\int_{\mathbb{R}^{N}}\dfrac{|u|^{2}}{|x|^{4m}}dx  &  \leq
\dfrac{N^{2}}{4}\int_{\mathbb{R}^{N}}\dfrac{|\nabla\Delta^{m-1}u|^{2}}%
{|x|^{2}}dx\\
&  \leq\int_{\mathbb{R}^{N}}|\Delta^{m}u|^{2}dx=A\text{,}%
\end{align*}
we deduce that with $v=\Delta^{m-1}u:$%
\begin{align*}
&  \left(  \int_{\mathbb{R}^{N}}|\Delta^{m}u|^{2}dx-C_{1}(N,m)\int
_{\mathbb{R}^{N}}\dfrac{|u|^{2}}{|x|^{4m}}dx\right)  ^{\theta}\left(
\int_{\mathbb{R}^{N}}\dfrac{|u|^{2}}{|x|^{4m}}dx\right)  ^{1-\theta}\\
&  \geq\left(  \int_{\mathbb{R}^{N}}|\Delta^{m}u|^{2}dx-\dfrac{N^{2}}{4}%
\int_{\mathbb{R}^{N}}\dfrac{|\nabla\Delta^{m-1}u|^{2}}{|x|^{2}}dx\right)
^{\theta}\left(  \int_{\mathbb{R}^{N}}\dfrac{|\nabla\Delta^{m-1}u|^{2}%
}{|x|^{2}}dx\right)  ^{1-\theta}\\
&  =\left(  \int_{\mathbb{R}^{N}}|\Delta v|^{2}dx-\dfrac{N^{2}}{4}%
\int_{\mathbb{R}^{N}}\dfrac{|\nabla v|^{2}}{|x|^{2}}dx\right)  ^{\theta
}\left(  \int_{\mathbb{R}^{N}}\dfrac{|\nabla v|^{2}}{|x|^{2}}dx\right)
^{1-\theta}.
\end{align*}
We will now show a stronger result that
\begin{equation}
\left(  \int_{\mathbb{R}^{N}}|\Delta v|^{2}dx-\dfrac{N^{2}}{4}\int
_{\mathbb{R}^{N}}\dfrac{|\nabla v|^{2}}{|x|^{2}}dx\right)  ^{\theta}\left(
\int_{\mathbb{R}^{N}}\dfrac{|\nabla v|^{2}}{|x|^{2}}dx\right)  ^{1-\theta
}\gtrsim\left\Vert \nabla v\right\Vert _{L^{\frac{2N}{N-2}}}^{2} \label{2.2.1}%
\end{equation}
for all radial functions $v$. Then, by combining the above estimates with the
Sobolev embedding $W^{2m-1,\frac{2N}{N-2}}\hookrightarrow L^{\frac{2N}{N-4m}}%
$, we obtain that
\begin{align*}
&  \left(  \int_{\mathbb{R}^{N}}|\Delta^{m}u|^{2}dx-C_{1}(N,m)\int
_{\mathbb{R}^{N}}\dfrac{|u|^{2}}{|x|^{4m}}dx\right)  ^{\theta}\left(
\int_{\mathbb{R}^{N}}\dfrac{|u|^{2}}{|x|^{4m}}dx\right)  ^{1-\theta}\\
&  \geq\left(  \int_{\mathbb{R}^{N}}|\Delta v|^{2}dx-\dfrac{N^{2}}{4}%
\int_{\mathbb{R}^{N}}\dfrac{|\nabla v|^{2}}{|x|^{2}}dx\right)  ^{\theta
}\left(  \int_{\mathbb{R}^{N}}\dfrac{|\nabla v|^{2}}{|x|^{2}}dx\right)
^{1-\theta}\\
&  \gtrsim\left\Vert \nabla v\right\Vert _{L^{\frac{2N}{N-2}}}^{2}\\
&  \gtrsim\left\Vert u\right\Vert _{L^{\frac{2N}{N-4m}}}^{2}.
\end{align*}

To prove (\ref{2.2.1}), we would like to compute
\[
\int_{\mathbb{R}^{N}}|\Delta v|^{2}dx-\dfrac{N^{2}}{4}\int_{\mathbb{R}^{N}%
}\dfrac{|\nabla v|^{2}}{|x|^{2}}dx
\]
in terms of $w$, where $w(r)=\frac{v^{\prime}(r)}{r}$. Since $v$ is radial, we
have
\begin{align*}
\int_{\mathbb{R}^{N}}|\Delta v|^{2}dx  &  =|\mathbb{S}^{N-1}|\int_{0}^{\infty
}\left\vert v^{\prime\prime}+\dfrac{N-1}{r}v^{\prime}\right\vert ^{2}%
r^{N-1}dr\\
&  =|\mathbb{S}^{N-1}|\left[  \int_{0}^{\infty}\left\vert v^{\prime\prime
}\right\vert ^{2}r^{N-1}dr+(N-1)\int_{0}^{\infty}\left\vert v^{\prime
}\right\vert ^{2}r^{N-3}dr\right] \\
&  =|\mathbb{S}^{N-1}|\left[  \int_{0}^{\infty}\left\vert w^{\prime
}\right\vert ^{2}r^{N+1}dr+(1-N)\int_{0}^{\infty}\left\vert w\right\vert
^{2}r^{N-1}dr+(N-1)\int_{0}^{\infty}\left\vert w\right\vert ^{2}%
r^{N-1}dr\right] \\
&  =|\mathbb{S}^{N-1}|\int_{0}^{\infty}\left\vert w^{\prime}\right\vert
^{2}r^{N+1}dr
\end{align*}
and
\[
\int_{\mathbb{R}^{N}}\dfrac{|\nabla v|^{2}}{|x|^{2}}dx=|\mathbb{S}^{N-1}%
|\int_{0}^{\infty}\left\vert v^{\prime}\right\vert ^{2}r^{N-3}dr=|\mathbb{S}%
^{N-1}|\int_{0}^{\infty}\left\vert w\right\vert ^{2}r^{N-1}dr.
\]
This implies that
\[
\int_{\mathbb{R}^{N}}|\Delta v|^{2}dx-\dfrac{N^{2}}{4}\int_{\mathbb{R}^{N}%
}\dfrac{|\nabla v|^{2}}{|x|^{2}}dx=|\mathbb{S}^{N-1}|\left[  \int_{0}^{\infty
}\left\vert w^{\prime}\right\vert ^{2}r^{N+1}dr-\frac{N^{2}}{4}\int
_{0}^{\infty}\left\vert w\right\vert ^{2}r^{N-1}dr\right]  .
\]
Let $f(r)=w(r)r^{\alpha}$, we have
\[
w^{\prime}(r)=\dfrac{f^{\prime}(r)}{r^{\alpha}}-\alpha\dfrac{f\left(
r\right)  }{r^{\alpha+1}}.
\]
Hence,
\begin{align*}
&  \int_{0}^{\infty}\left\vert w^{\prime}\right\vert ^{2}r^{N+1}dr\\
&  =\int_{0}^{\infty}\left\vert f^{\prime}\right\vert ^{2}r^{N-2\alpha
+1}dr+\alpha^{2}\int_{0}^{\infty}\left\vert f\right\vert ^{2}r^{N-2\alpha
-1}dr-2\alpha\int_{0}^{\infty}ff^{\prime}r^{N-2\alpha}dr\\
&  =\int_{0}^{\infty}\left\vert f^{\prime}\right\vert ^{2}r^{N-2\alpha
+1}dr+\alpha^{2}\int_{0}^{\infty}\left\vert f\right\vert ^{2}r^{N-2\alpha
-1}dr+\alpha(N-2\alpha)\int_{0}^{\infty}\left\vert f\right\vert ^{2}%
r^{N-2\alpha-1}dr\\
&  =\int_{0}^{\infty}\left\vert f^{\prime}\right\vert ^{2}r^{N-2\alpha
+1}dr+\alpha(N-\alpha)\int_{0}^{\infty}\left\vert f\right\vert ^{2}%
r^{N-2\alpha-1}dr,
\end{align*}
and
\[
\int_{0}^{\infty}\left\vert w\right\vert ^{2}r^{N-1}dr=\int_{0}^{\infty
}\left\vert f\right\vert ^{2}r^{N-2\alpha-1}dr.
\]
Therefore,
\begin{align*}
&  \int_{\mathbb{R}^{N}}|\Delta v|^{2}dx-\dfrac{N^{2}}{4}\int_{\mathbb{R}^{N}%
}\dfrac{|\nabla v|^{2}}{|x|^{2}}dx\\
&  =|\mathbb{S}^{N-1}|\left[  \int_{0}^{\infty}\left\vert w^{\prime
}\right\vert ^{2}r^{N+1}dr-\dfrac{N^{2}}{4}\int_{0}^{\infty}\left\vert
w\right\vert ^{2}r^{N-1}dr\right] \\
&  =|\mathbb{S}^{N-1}|\left[  \int_{0}^{\infty}\left\vert f^{\prime
}\right\vert ^{2}r^{N-2\alpha+1}dr+\left(  \alpha(N-\alpha)-\dfrac{N^{2}}%
{4}\right)  \int_{0}^{\infty}\left\vert f\right\vert ^{2}r^{N-2\alpha
-1}dr\right] \\
&  =|\mathbb{S}^{N-1}|\left[  \int_{0}^{\infty}\left\vert f^{\prime
}\right\vert ^{2}r^{N-2\alpha+1}dr-\left(  \alpha-\dfrac{N}{2}\right)
^{2}\int_{0}^{\infty}\left\vert f\right\vert ^{2}r^{N-2\alpha-1}dr\right]  .
\end{align*}
Thus, by choosing $\alpha=N/2$, we obtain
\[
\int_{\mathbb{R}^{N}}|\Delta v|^{2}dx-\dfrac{N^{2}}{4}\int_{\mathbb{R}^{N}%
}\dfrac{|\nabla v|^{2}}{|x|^{2}}dx=|\mathbb{S}^{N-1}|\int_{0}^{\infty
}r\left\vert f^{\prime}\right\vert ^{2}dr,
\]
where $v^{\prime}=\frac{f}{r^{N/2-1}}.$ As a result,
\[
\text{LHS}_{\eqref{2.2.1}}=|\mathbb{S}^{N-1}|\left(  \int_{0}^{\infty
}r\left\vert f^{\prime}\right\vert ^{2}dr\right)  ^{\theta}\left(  \int
_{0}^{\infty}\dfrac{\left\vert f\right\vert ^{2}}{r}dr\right)  ^{1-\theta}.
\]
Then, it suffices to prove that
\[
\left(  \int_{0}^{\infty}r\left\vert f^{\prime}\right\vert ^{2}dr\right)
^{\theta}\left(  \int_{0}^{\infty}\dfrac{|f(r)|^{2}}{r}dr\right)  ^{1-\theta
}\geq C_{rad,2^{\ast}}|\mathbb{S}^{d-1}|^{2/2^{\ast}-1}\left(  \int
_{0}^{\infty}\dfrac{|f(r)|^{2^{\ast}}}{r}dr\right)  ^{2/2^{\ast}},
\]
which is exactly what we have already received from \cite{DN23}. Here
$2^{\ast}=\frac{2N}{N-2}$. Also from \cite{DN23}, we also need the condition
\[
\theta=1/N
\]
so that the above estimate holds.

If $(1-\theta)A\geq B$, that is $A-B\geq\theta A$, then it is enough to show
that
\[
\left(  \int_{\mathbb{R}^{N}}|\Delta^{m}u|^{2}dx\right)  ^{\theta}\left(
\int_{\mathbb{R}^{N}}\dfrac{|u|^{2}}{|x|^{4m}}dx\right)  ^{1-\theta}\geq
C\left\Vert u\right\Vert _{L^{\frac{2N}{N-4m}}}^{2}.
\]
By Theorem \ref{T3.2}, we have that with $u\left(  \cdot\right)  =U\left(
\left\vert \cdot\right\vert \right)  :$
\[
\left(  \int_{\mathbb{R}^{N}}|\Delta^{m}u|^{2}dx\right)  ^{\theta}\left(
\int_{\mathbb{R}^{N}}\dfrac{|u|^{2}}{|x|^{4m}}dx\right)  ^{1-\theta}%
\gtrsim\left(  \int_{0}^{\infty}\left\vert U^{\left(  2m\right)  }\right\vert
^{2}r^{N-1}dr\right)  ^{\theta}\left(  \int_{0}^{\infty}|U|^{2}r^{N-1-4m}%
dr\right)  ^{1-\theta}.
\]
Therefore, it is enough to show that
\[
\left(  \int_{0}^{\infty}\left\vert U^{\left(  2m\right)  }\right\vert
^{2}r^{N-1}dr\right)  ^{\theta}\left(  \int_{0}^{\infty}|U|^{2}r^{N-1-4m}%
dr\right)  ^{1-\theta}\gtrsim\left(  \int_{0}^{\infty}|U|^{\frac{2N}{N-4m}%
}r^{N-1}dr\right)  ^{\frac{N-4m}{N}}.
\]
However, this is just a consequence of the higher order
Caffarelli-Kohn-Nirenberg inequality (Proposition \ref{HCKN}) in $1$
dimension. Indeed, in our case, $n=1$, $k=2m$, $j=0$, $p=q=2,$ $s=\frac
{2N}{N-4m}$, $a=\theta$, $\alpha=\frac{N-1}{2}$, $\beta=\frac{N-1-4m}{2}$,
$\gamma=\frac{\left(  N-1\right)  \left(  N-4m\right)  }{2N}$.

\textbf{Case 2:} $\theta=\frac{2m}{N}$.

By the Hardy-Rellich inequalities (Theorem \ref{Thm 1.3}), we have
\begin{align*}
&  \left(  \int_{\mathbb{R}^{N}}|\Delta^{m}u|^{2}dx-C_{1}(N,m)\int
_{\mathbb{R}^{N}}\dfrac{|u|^{2}}{|x|^{4m}}dx\right)  ^{\theta}\left(
\int_{\mathbb{R}^{N}}\dfrac{|u|^{2}}{|x|^{4m}}dx\right)  ^{1-\theta}\\
&  \gtrsim\left(  \int_{\mathbb{R}^{N}}\frac{|\nabla u|^{2}}{|x|^{4m-2}%
}dx-\frac{\left(  N-4m\right)  ^{2}}{4}\int_{\mathbb{R}^{N}}\dfrac{|u|^{2}%
}{|x|^{4m}}dx\right)  ^{\theta}\left(  \int_{\mathbb{R}^{N}}\dfrac{|u|^{2}%
}{|x|^{4m}}dx\right)  ^{1-\theta}.
\end{align*}
Therefore, it suffices to prove that
\begin{equation}
\left(  \int_{\mathbb{R}^{N}}\dfrac{|\nabla u|^{2}}{|x|^{4m-2}}dx-\dfrac
{(N-4m)^{2}}{4}\int_{\mathbb{R}^{N}}\dfrac{|u|^{2}}{|x|^{4m}}dx\right)
^{\theta}\left(  \int_{\mathbb{R}^{N}}\dfrac{|u|^{2}}{|x|^{4m}}dx\right)
^{1-\theta}\geq C\left\Vert u\right\Vert _{L^{\frac{2N}{N-4m}}}^{2}.
\label{eq 2.28}%
\end{equation}
Since $u$ is radial,
\begin{align*}
&  \int_{\mathbb{R}^{N}}\dfrac{|\nabla u|^{2}}{|x|^{4m-2}}dx-\dfrac
{(N-4m)^{2}}{4}\int_{\mathbb{R}^{N}}\dfrac{|u|^{2}}{|x|^{4m}}dx\\
&  =|\mathbb{S}^{N-1}|\left[  \int_{0}^{\infty}\left\vert u^{\prime
}\right\vert ^{2}r^{N-4m+1}dr-\dfrac{(N-4m)^{2}}{4}\int_{0}^{\infty}%
|u|^{2}r^{N-4m-1}dr\right]  .
\end{align*}
By letting $f(r)=u(r).r^{\alpha}$, we get
\begin{align*}
&  \int_{0}^{\infty}\left\vert u^{\prime}\right\vert ^{2}r^{N-4m+1}dr\\
&  =\int_{0}^{\infty}\left\vert f^{\prime}\right\vert ^{2}r^{N-4m-2\alpha
+1}dr+\alpha^{2}\int_{0}^{\infty}\left\vert f\right\vert ^{2}r^{N-4m-2\alpha
-1}dr-2\alpha\int_{0}^{\infty}ff^{\prime}r^{N-4m-2\alpha}dr\\
&  =\int_{0}^{\infty}\left\vert f^{\prime}\right\vert ^{2}r^{N-4m-2\alpha
+1}dr+\alpha^{2}\int_{0}^{\infty}\left\vert f\right\vert ^{2}r^{N-4m-2\alpha
-1}dr+\alpha(N-4m-2\alpha)\int_{0}^{\infty}\left\vert f\right\vert
^{2}r^{N-4m-2\alpha-1}dr\\
&  =\int_{0}^{\infty}\left\vert f^{\prime}\right\vert ^{2}r^{N-4m-2\alpha
+1}dr+\alpha(N-4m-\alpha)\int_{0}^{\infty}\left\vert f\right\vert
^{2}r^{N-4m-2\alpha-1}dr,
\end{align*}
and
\[
\int_{0}^{\infty}\left\vert u\right\vert ^{2}r^{N-4m-1}dr=\int_{0}^{\infty
}\left\vert f\right\vert ^{2}r^{N-4m-2\alpha-1}dr.
\]
It allows us
\begin{align*}
&  \int_{\mathbb{R}^{N}}\dfrac{|\nabla u|^{2}}{|x|^{4m-2}}dx-\dfrac
{(N-4m)^{2}}{4}\int_{\mathbb{R}^{N}}\dfrac{|u|^{2}}{|x|^{4m}}dx\\
&  =|\mathbb{S}^{N-1}|\left[  \int_{0}^{\infty}\left\vert f^{\prime
}\right\vert ^{2}r^{N-4m-2\alpha+1}dr+\left(  \alpha(N-4m-\alpha
)-\dfrac{(N-4m)^{2}}{4}\right)  \int_{0}^{\infty}\left\vert f\right\vert
^{2}r^{N-4m-2\alpha-1}dr\right] \\
&  =|\mathbb{S}^{N-1}|\left[  \int_{0}^{\infty}\left\vert f^{\prime
}\right\vert ^{2}r^{N-4m-2\alpha+1}dr-\left(  \alpha-\dfrac{N-4m}{2}\right)
^{2}\int_{0}^{\infty}\left\vert f\right\vert ^{2}r^{N-4m-2\alpha-1}dr\right]
.
\end{align*}
Thus, choosing $\alpha=\dfrac{N-4m}{2}$, we obtain
\[
\int_{\mathbb{R}^{N}}\dfrac{|\nabla u|^{2}}{|x|^{4m-2}}dx-\dfrac{(N-4m)^{2}%
}{4}\int_{\mathbb{R}^{N}}\dfrac{|u|^{2}}{|x|^{4m}}dx=|\mathbb{S}^{N-1}%
|\int_{0}^{\infty}r\left\vert f^{\prime}\right\vert ^{2}dr,
\]
where $u=\frac{f}{r^{(N-4m)/2}}.$ As a result,
\[
\text{LHS}_{\eqref{eq 2.28}}=|\mathbb{S}^{N-1}|\left(  \int_{0}^{\infty
}r\left\vert f^{\prime}\right\vert ^{2}dr\right)  ^{\theta}\left(  \int
_{0}^{\infty}\dfrac{\left\vert f\right\vert ^{2}}{r}dr\right)  ^{1-\theta}.
\]
Then, it remains to prove that
\[
\left(  \int_{0}^{\infty}r\left\vert f^{\prime}\right\vert ^{2}dr\right)
^{\theta}\left(  \int_{0}^{\infty}\dfrac{|f(r)|^{2}}{r}dr\right)  ^{1-\theta
}\geq C_{rad}|\mathbb{S}^{N-1}|^{\frac{N-4m}{N}-1}\left(  \int_{0}^{\infty
}\dfrac{|f(r)|^{\frac{2N}{N-4m}}}{r}dr\right)  ^{\frac{N-4m}{N}},
\]
which is equivalent to
\[
\left(  \int_{\mathbb{R}}\left\vert \psi^{\prime}\right\vert ^{2}ds\right)
^{\theta}\left(  \int_{\mathbb{R}}|\psi(s)|^{2}\right)  ^{1-\theta}\geq
C_{rad}|\mathbb{S}^{N-1}|^{\frac{N-4m}{N}-1}\left(  \int_{\mathbb{R}}%
|\psi(s)|^{\frac{2N}{N-4m}}\right)  ^{\frac{N-4m}{N}},
\]
where $f(r)=\psi(\log r)=\psi(s)$. This inequality is a direct consequence of
the main theorem in \cite{N41} with $\alpha=p=2$ and $\beta=\dfrac{8m}{N-4m}$.
In this case, we need the condition $\theta=\dfrac{2m}{N}$.

\textbf{Case 3:} $\frac{1}{N}<\theta<\frac{2m}{N}$. This case is just a
consequence of Case 1 and Case 2 via a simple interpolation argument.

By the same approach, we also obtain (\ref{odd-1i}).
\end{proof}

\section{Proof of Theorem \ref{T1}}

Theorem \ref{T1} can be deduced from the following Theorem \ref{T2.1}, Theorem
\ref{Tg1} and Theorem \ref{Tg2}.

To illustriate our main ideas in proving Theorem \ref{T1}, we will start with
the proof of Theorem \ref{T1} in the special case $N=1/\theta=5$. More
precisely, we will prove that

\begin{theorem}
\label{T2.1} With $N=1/\theta=5$, we have
\begin{equation}
\left(  \int_{\mathbb{R}^{N}}|\Delta u|^{2}dx-\dfrac{N^{2}(N-4)^{2}}{16}%
\sup_{y\in\mathbb{R}^{N}}\int_{\mathbb{R}^{N}}\dfrac{|u(x)|^{2}}{|x-y|^{4}%
}dx\right)  ^{\theta}\left(  \sup_{y\in\mathbb{R}^{N}}\int_{\mathbb{R}^{N}%
}\dfrac{|u(x)|^{2}}{|x-y|^{4}}dx\right)  ^{1-\theta}\geq C\left\Vert
u\right\Vert _{L^{\frac{2N}{N-4}}}^{2} \label{2-0oi}%
\end{equation}
holds with a constant $C=C(N,\theta)>0$ independent of $u\in D^{2,2}%
(\mathbb{R}^{N})$.
\end{theorem}

\begin{proof}
[Proof of Theorem \ref{T2.1}]Let $u\in C_{0}^{\infty}\left(  \mathbb{R}%
^{N}\right)  $ and denote
\[
A=\int_{\mathbb{R}^{N}}|\Delta u|^{2}dx,\;\;B=\dfrac{N^{2}(N-4)^{2}}{16}%
\sup_{y\in\mathbb{R}^{N}}\int_{\mathbb{R}^{N}}\dfrac{|u(x)|^{2}}{|x-y|^{4}%
}dx.
\]
Obviously, by the Rellich inequality (Theorem \ref{Thm 1.3}), we have that
$A\geq B\geq0$.

\textbf{Case 1:} $(1-\theta)A\geq B$\textbf{.} In this case, we have
\[
B\geq\dfrac{N^{2}(N-4)^{2}}{16}\sup_{r>0,y\in\mathbb{R}^{N}}\int
_{B(y,r)}\dfrac{|u(x)|^{2}}{|x-y|^{4}}dx\geq\dfrac{N^{2}(N-4)^{2}}{16}%
\sup_{r>0,y\in\mathbb{R}^{N}}\dfrac{1}{r^{4}}\int_{B(y,r)}|u(x)|^{2}dx.
\]
Moreover, since $A-B\geq\theta A$, it follows that
\begin{align*}
(A-B)^{\theta}B^{1-\theta}  &  \gtrsim\left\Vert \Delta u\right\Vert _{L^{2}%
}^{2\theta}\left(  \sup_{r>0,y\in\mathbb{R}^{N}}\dfrac{1}{r^{4}}\int
_{B(y,r)}|u(x)|^{2}dx\right)  ^{1-\theta}\\
&  \gtrsim\left\Vert \Delta u\right\Vert _{L^{2}}^{2\theta}\left\Vert
u\right\Vert _{\mathbf{M}^{2,N-4}}^{2(1-\theta)}.
\end{align*}
Then, it suffices to prove that
\begin{equation}
\left\Vert \Delta u\right\Vert _{L^{2}}^{2\theta}\left\Vert u\right\Vert
_{\mathbf{M}^{2,N-4}}^{2(1-\theta)}\gtrsim\left\Vert u\right\Vert
_{L^{\frac{2N}{N-4}}}^{2}. \label{eq 2.2}%
\end{equation}
Applying Proposition \ref{lem 2.2} and Theorem \ref{Thm 1.2} with
$k=0,r=2^{\ast\ast}=\frac{2N}{N-4}$, $m=2$ and $p=2$, we have
\[
\frac{2N}{N-4}=\frac{2(2+s)}{s},
\]
which implies that $s=\frac{N-4}{2}$, and
\begin{align*}
\left\Vert u\right\Vert _{L^{2^{\ast\ast}}}  &  \lesssim\left\Vert
u\right\Vert _{\dot{B}^{-s}}^{4/N}\left\Vert \Delta u\right\Vert _{L^{2}%
}^{(N-4)/N}\\
&  \lesssim\left\Vert u\right\Vert _{M^{1,s}}^{4/N}\left\Vert \Delta
u\right\Vert _{L^{2}}^{(N-4)/N}\\
&  \lesssim\left\Vert u\right\Vert _{M^{2,2s}}^{4/N}\left\Vert \Delta
u\right\Vert _{L^{2}}^{(N-4)/N}.
\end{align*}
The last one is exactly \eqref{eq 2.2} with $\theta=\dfrac{N-4}{N}$. Combining
this with the Sobolev inequality $\left\Vert u\right\Vert _{L^{2^{\ast\ast}}%
}\lesssim\left\Vert \Delta u\right\Vert _{L^{2}}$, we obtain \eqref{eq 2.2}
with $\theta\geq\dfrac{N-4}{N}.$

\textbf{Case 2:} $(1-\theta)A<B$. We note that since
\[
\dfrac{d}{dB}((A-B)^{\theta}B^{1-\theta})=((1-\theta)A-B)(A-B)^{1-\theta
}B^{-\theta}%
\]
the function $(A-B)^{\theta}B^{1-\theta}$ is non-increasing with respect to
$B$ if $0\leq(1-\theta)A<B\leq A$.

Let $\Omega=supp\left(  u\right)  $ and $f=-\Delta u$. Let $\Omega^{\ast}$ the
ball centered at the origin such that $\left\vert \Omega^{\ast}\right\vert
=\left\vert \Omega\right\vert $ and $f^{\ast}$ the standard Schwarz
rearrangement of $f$. We consider the following Dirichlet problem%
\[
\left\{
\begin{array}
[c]{cc}%
-\Delta v=f^{\ast} & \text{in }\Omega^{\ast}\\
v=0 & \text{on }\partial\Omega^{\ast}%
\end{array}
\right.  .
\]
Then we have the following comparison principle (see, for instance,
\cite{Tal76}): $u^{\ast}\leq v$. Therefore, by the Hardy-Littlewood
inequality, we have
\[
\sup_{y\in\mathbb{R}^{N}}\int_{\mathbb{R}^{N}}\dfrac{|u(x)|^{2}}{|x-y|^{4}%
}dx\leq\int_{\mathbb{R}^{N}}\dfrac{|u^{\ast}|^{2}}{|x|^{4}}dx\leq
\int_{\mathbb{R}^{N}}\dfrac{|v|^{2}}{|x|^{4}}dx.
\]
We also note that
\[
A=\int_{\mathbb{R}^{N}}|\Delta u|^{2}dx=\int_{\mathbb{R}^{N}}|f|^{2}%
dx=\int_{\mathbb{R}^{N}}|f^{\ast}|^{2}dx=\int_{\mathbb{R}^{N}}|\Delta
v|^{2}dx.
\]
Therefore by the Rellich inequality, we have that
\[
(1-\theta)A<B\leq\dfrac{N^{2}(N-4)^{2}}{16}\int_{\mathbb{R}^{N}}\dfrac
{|v|^{2}}{|x|^{4}}dx\leq A.
\]
Since the function $x\mapsto(A-x)^{\theta}x^{1-\theta}$ is non-increasing on
$\left[  (1-\theta)A,A\right]  $, we get%
\begin{align*}
&  \left(  \int_{\mathbb{R}^{N}}|\Delta u|^{2}dx-\dfrac{N^{2}(N-4)^{2}}%
{16}\sup_{y\in\mathbb{R}^{N}}\int_{\mathbb{R}^{N}}\dfrac{|u(x)|^{2}}%
{|x-y|^{4}}dx\right)  ^{\theta}\left(  \sup_{y\in\mathbb{R}^{N}}%
\int_{\mathbb{R}^{N}}\dfrac{|u(x)|^{2}}{|x-y|^{4}}dx\right)  ^{1-\theta}\\
&  \geq\left(  \int_{\mathbb{R}^{N}}|\Delta v|^{2}dx-\dfrac{N^{2}(N-4)^{2}%
}{16}\int_{\mathbb{R}^{N}}\dfrac{|v|^{2}}{|x|^{4}}dx\right)  ^{\theta}\left(
\int_{\mathbb{R}^{N}}\dfrac{|v|^{2}}{|x|^{4}}dx\right)  ^{1-\theta}.
\end{align*}
Also, since $\left\Vert u\right\Vert _{L^{\frac{2N}{N-4}}}^{2}=\left\Vert
u^{\ast}\right\Vert _{L^{\frac{2N}{N-4}}}^{2}\leq\left\Vert v\right\Vert
_{L^{\frac{2N}{N-4}}}^{2}$, it is now enough to show that
\begin{equation}
\left(  \int_{\mathbb{R}^{N}}|\Delta v|^{2}dx-\dfrac{N^{2}(N-4)^{2}}{16}%
\int_{\mathbb{R}^{N}}\dfrac{|v|^{2}}{|x|^{4}}dx\right)  ^{\theta}\left(
\int_{\mathbb{R}^{N}}\dfrac{|v|^{2}}{|x|^{4}}dx\right)  ^{1-\theta}%
\gtrsim\left\Vert v\right\Vert _{L^{\frac{2N}{N-4}}}^{2} \label{2.1}%
\end{equation}
for all radial functions $v$. Obvisouly, this is just a consequence of Theorem
\ref{T3}. Also, as noted in the proof of Theorem \ref{T3}, we also need the
condition
\[
\theta=1/N
\]
so that the above estimate holds.

Finally, we note that by combining the two conditions $\theta\geq\dfrac
{N-4}{N}$ and $\theta=\frac{1}{N}$, we get $N=5$ and $\theta=\frac{1}{5}$.
\end{proof}

Next, we state here a comparison principle for higher order derivatives, which
is a generalization of the comparison principle in \cite{Tal76}. Indeed,

\begin{proposition}
\label{prop 2.3} Let us consider two problems
\[%
\begin{cases}
(-\Delta)^{k} u=f & \text{ in } \Omega\\
(-\Delta)^{i} u=0 & \text{ on } \partial\Omega
\end{cases}
\]
for all $i=\overline{0, k-1}$, and
\[%
\begin{cases}
(-\Delta)^{k} v=f^{*} & \text{ in } \Omega^{*}\\
(-\Delta)^{i} v=0 & \text{ on } \partial\Omega^{*}%
\end{cases}
\]
for all $i=\overline{0, k-1}$. Assume that $u$ and $v$ are weak solutions to
the two above problems, then $u^{*} \leq v$ in $\Omega^{*}$.
\end{proposition}

See, for instance, \cite{Tar12}. For the convenience of the reader, we will
present here a proof of this result.

\begin{proof}
[Proof of Proposition \ref{prop 2.3}]For convenience, we denote $(-\Delta
)^{i}u=u_{i}$ and $(-\Delta)^{i}v=v_{i}$ for all $i=\overline{0,k}$. Using the
Talenti's comparison principle \cite{Tal76}, we have
\[
u_{k-1}^{\ast}\leq v_{k-1}.
\]
The idea is to use this comparison principle consecutively to obtain
$u_{i}^{\ast}\leq v_{i}$ for all $i=\overline{0,k}$. To do this, assume that
$u_{j}^{\ast}\leq v_{j}$ for some $j=\overline{1,k-1}$. We claim that
\[
u_{j-1}^{\ast}\leq v_{j-1}.
\]
More clearly, by definition, $u_{j-1}$ and $v_{j-1}$ are solutions of
\[%
\begin{cases}
-\Delta u_{j-1}=u_{j} & \text{ in }\Omega\\
u_{j-1}=0 & \text{ on }\partial\Omega
\end{cases}
,
\]
and
\[%
\begin{cases}
-\Delta v_{j-1}=v_{j} & \text{ in }\Omega^{\ast}\\
v_{j-1}=0 & \text{ on }\partial\Omega^{\ast}%
\end{cases}
,
\]
respectively. Let us consider one more problem
\[%
\begin{cases}
-\Delta w=u_{j}^{\ast} & \text{ in }\Omega^{\ast}\\
w=0 & \text{ on }\partial\Omega^{\ast}%
\end{cases}
.
\]
By applying the Talenti's comparison principle for $u_{j-1}$ and $w$, we have
$u_{j-1}^{\ast}\leq w$. Moreover, $w-v_{j-1}$ is the solution of
\[%
\begin{cases}
-\Delta(w-v_{j-1})=u_{j}^{\ast}-v_{j} & \text{ in }\Omega^{\ast}\\
w-v_{j-1}=0 & \text{ on }\partial\Omega^{\ast}%
\end{cases}
.
\]
Since $u_{j}^{\ast}-v_{j}\leq0$, by the maximum principle, we get
$w-v_{j-1}\leq0$, which implies that $u_{j-1}^{\ast}\leq w\leq v_{j-1}$, as
desired. Therefore, by an induction argument, $u^{\ast}\leq v$ in
$\Omega^{\ast}$.
\end{proof}

We are now ready to provide a proof for Theorem \ref{T1}. More precisely, we have

\begin{theorem}
\label{Tg1} With $N=1/\theta=4m+1$ $(m\in\mathbb{Z}_{+})$ and $C_{1}(N,m)$
defined in Theorem \ref{Thm 1.3}, we have
\begin{equation}
\left(  \int_{\mathbb{R}^{N}}|\Delta^{m}u|^{2}dx-C_{1}(N,m)\sup_{y\in
\mathbb{R}^{N}}\int_{\mathbb{R}^{N}}\dfrac{|u(x)|^{2}}{|x-y|^{4m}}dx\right)
^{\theta}\left(  \sup_{y\in\mathbb{R}^{N}}\int_{\mathbb{R}^{N}}\dfrac
{|u(x)|^{2}}{|x-y|^{4m}}dx\right)  ^{1-\theta}\geq C\left\Vert u\right\Vert
_{L^{\frac{2N}{N-4m}}}^{2} \label{2m-0i1}%
\end{equation}
holds with a constant $C=C(N,\theta)>0$ independent of $u\in D^{2m,2}%
(\mathbb{R}^{N})$.
\end{theorem}

\begin{proof}
Let $u\in C_{0}^{\infty}\left(  \mathbb{R}^{N}\right)  $ and denote
\[
A=\int_{\mathbb{R}^{N}}|\Delta^{m}u|^{2}dx,\;\;B=C_{1}(N,m)\sup_{y\in
\mathbb{R}^{N}}\int_{\mathbb{R}^{N}}\dfrac{|u(x)|^{2}}{|x-y|^{4m}}dx.
\]
Note that by the Rellich inequality (Theorem \ref{Thm 1.3}), we have that
$A\geq B$.

\textbf{Case 1:} $(1-\theta)A\geq B$\textbf{.} In this case, we have
\[
B\geq C_{1}(N,m)\sup_{r>0,y\in\mathbb{R}^{N}}\int_{B(y,r)}\dfrac{|u(x)|^{2}%
}{|x-y|^{4m}}dx\geq C_{1}(N,m)\sup_{r>0,y\in\mathbb{R}^{N}}\dfrac{1}{r^{4m}%
}\int_{B(y,r)}|u(x)|^{2}dx.
\]
Moreover, since $A-B\geq\theta A$, it follows that
\begin{align*}
(A-B)^{\theta}B^{1-\theta}  &  \gtrsim\left\Vert \Delta^{m}u\right\Vert
_{L^{2}}^{2\theta}\left(  \sup_{r>0,y\in\mathbb{R}^{N}}\dfrac{1}{r^{4m}}%
\int_{B(y,r)}|u(x)|^{2}dx\right)  ^{1-\theta}\\
&  \gtrsim\left\Vert \Delta^{m}u\right\Vert _{L^{2}}^{2\theta}\left\Vert
u\right\Vert _{\mathbf{M}^{2,N-4m}}^{2(1-\theta)}.
\end{align*}
Then, it suffices to prove that
\begin{equation}
\left\Vert \Delta^{m}u\right\Vert _{L^{2}}^{2\theta}\left\Vert u\right\Vert
_{\mathbf{M}^{2,N-4m}}^{2(1-\theta)}\gtrsim\left\Vert u\right\Vert
_{L^{2N/(N-4m)}}^{2}. \label{eq 2.10}%
\end{equation}
Applying Theorem \ref{Thm 1.2} with $k=0,r=\frac{2N}{N-4m}$, $m=2m$ and $p=2$,
we have
\[
\frac{2N}{N-4m}=\frac{2(2m+s)}{s},
\]
which implies that $s=\frac{N-4m}{2}$, and
\begin{align*}
\left\Vert u\right\Vert _{L^{2N/(N-4m)}}  &  \lesssim\left\Vert u\right\Vert
_{\dot{B}^{-s}}^{4m/N}\left\Vert \Delta^{m}u\right\Vert _{L^{2}}^{(N-4m)/N}\\
&  \lesssim\left\Vert u\right\Vert _{M^{1,s}}^{4m/N}\left\Vert \Delta
^{m}u\right\Vert _{L^{2}}^{(N-4m)/N}\\
&  \lesssim\left\Vert u\right\Vert _{M^{2,2s}}^{4m/N}\left\Vert \Delta
^{m}u\right\Vert _{L^{2}}^{(N-4m)/N},
\end{align*}
where the second and third estimate are from Proposition \ref{lem 2.2}. The
last one is exactly \eqref{eq 2.10} with $\theta=\dfrac{N-4m}{N}$. Combining
this with the Sobolev inequality $\left\Vert u\right\Vert _{L^{2N/(N-4m)}%
}\lesssim\left\Vert \Delta^{m}u\right\Vert _{L^{2}}$, we get \eqref{eq 2.10}
with $\theta\geq\dfrac{N-4m}{N}.$

\textbf{Case 2:} $(1-\theta)A<B$. Recall that the function $(A-B)^{\theta
}B^{1-\theta}$ is non-increasing with respect to $B$ if $0\leq(1-\theta
)A<B\leq A$.

Let $\Omega=supp\left(  u\right)  $ and $f=(-\Delta)^{m}u$, then $u$ is the
solution of the problem
\[%
\begin{cases}
(-\Delta)^{m}u=f & \text{ in }\Omega\\
(-\Delta)^{i}u=0 & \text{ on }\partial\Omega
\end{cases}
\]
for all $i=0,1,\ldots,m-1$. Let $\Omega^{\ast}$ the ball centered at the
origin such that $\left\vert \Omega^{\ast}\right\vert =\left\vert
\Omega\right\vert $ and $f^{\ast}$ the standard Schwarz rearrangement of $f$.
We also consider the following Dirichlet problem
\[
\left\{
\begin{array}
[c]{cc}%
(-\Delta)^{m}v=f^{\ast} & \text{in }\Omega^{\ast}\\
(-\Delta)^{i}v=0 & \text{on }\partial\Omega^{\ast}%
\end{array}
\right.
\]
for all $i=0,1,\ldots,m-1$. Then, from Proposition \ref{prop 2.3} $u^{\ast
}\leq v$. Therefore, by the Hardy-Littlewood inequality, we have
\[
\sup_{y\in\mathbb{R}^{N}}\int_{\mathbb{R}^{N}}\dfrac{|u(x)|^{2}}{|x-y|^{4m}%
}dx\leq\int_{\mathbb{R}^{N}}\dfrac{|u^{\ast}|^{2}}{|x|^{4m}}dx\leq
\int_{\mathbb{R}^{N}}\dfrac{|v|^{2}}{|x|^{4m}}dx.
\]
We also note that
\[
A=\int_{\mathbb{R}^{N}}|\Delta^{m}u|^{2}dx=\int_{\mathbb{R}^{N}}|f|^{2}%
dx=\int_{\mathbb{R}^{N}}|f^{\ast}|^{2}dx=\int_{\mathbb{R}^{N}}|\Delta
^{m}v|^{2}dx.
\]
Therefore by Theorem \ref{Thm 1.3}, we have that
\[
(1-\theta)A<B\leq C_{1}(N,m)\int_{\mathbb{R}^{N}}\dfrac{|v|^{2}}{|x|^{4m}%
}dx\leq A.
\]
Since the function $(A-x)^{\theta}x^{1-\theta}$ is non-increasing on $\left[
(1-\theta)A,A\right]  $, we get
\begin{align*}
&  \left(  \int_{\mathbb{R}^{N}}|\Delta^{m}u|^{2}dx-C_{1}(N,m)\sup
_{y\in\mathbb{R}^{N}}\int_{\mathbb{R}^{N}}\dfrac{|u(x)|^{2}}{|x-y|^{4m}%
}dx\right)  ^{\theta}\left(  \sup_{y\in\mathbb{R}^{N}}\int_{\mathbb{R}^{N}%
}\dfrac{|u(x)|^{2}}{|x-y|^{4m}}dx\right)  ^{1-\theta}\\
&  \geq\left(  \int_{\mathbb{R}^{N}}|\Delta^{m}v|^{2}dx-C_{1}(N,m)\int
_{\mathbb{R}^{N}}\dfrac{|v|^{2}}{|x|^{4m}}dx\right)  ^{\theta}\left(
\int_{\mathbb{R}^{N}}\dfrac{|v|^{2}}{|x|^{4m}}dx\right)  ^{1-\theta}.
\end{align*}
Also, note that $\left\Vert u\right\Vert _{L^{\frac{2N}{N-4m}}}^{2}=\left\Vert
u^{\ast}\right\Vert _{L^{\frac{2N}{N-4m}}}^{2}\leq\left\Vert v\right\Vert
_{L^{\frac{2N}{N-4m}}}^{2}$. Therefore, it is now enough to show that
\begin{equation}
\left(  \int_{\mathbb{R}^{N}}|\Delta^{m}u|^{2}dx-C_{1}(N,m)\int_{\mathbb{R}%
^{N}}\dfrac{|u|^{2}}{|x|^{4m}}dx\right)  ^{\theta}\left(  \int_{\mathbb{R}%
^{N}}\dfrac{|u|^{2}}{|x|^{4m}}dx\right)  ^{1-\theta}\gtrsim\left\Vert
u\right\Vert _{L^{\frac{2N}{N-4m}}}^{2} \label{2.11}%
\end{equation}
for all radial functions $u$, which is just a consequence of Theorem \ref{T3}.
Noting that, in this case, we need $\theta=1/N$. Thus, we complete the proof.
\end{proof}

\begin{theorem}
\label{Tg2} With $N=1/\theta=4m+3$ and $C_{2}(N,m)$ defined in Theorem
\ref{Thm 1.3}, we have
\begin{align}
&  \left(  \int_{\mathbb{R}^{N}}|\nabla\Delta^{m}u|^{2}dx-C_{2}(N,m)\sup
_{y\in\mathbb{R}^{N}}\int_{\mathbb{R}^{N}}\dfrac{|u(x)|^{2}}{|x-y|^{4m+2}%
}dx\right)  ^{\theta}\left(  \sup_{y\in\mathbb{R}^{N}}\int_{\mathbb{R}^{N}%
}\dfrac{|u(x)|^{2}}{|x-y|^{4m+2}}dx\right)  ^{1-\theta}\nonumber\\
&  \geq C\left\Vert u\right\Vert _{L^{\frac{2N}{N-4m-2}}}^{2} \label{odd-0i1}%
\end{align}
holds with a constant $C=C(N,\theta)>0$ independent of $u\in D^{2m+1,2}%
(\mathbb{R}^{N})$.
\end{theorem}

\begin{proof}
Let $u\in C_{0}^{\infty}\left(  \mathbb{R}^{N}\right)  $ and denote
\[
A=\int_{\mathbb{R}^{N}}|\nabla\Delta^{m}u|^{2}dx,\;\;B=C_{2}(N,m)\sup
_{y\in\mathbb{R}^{N}}\int_{\mathbb{R}^{N}}\dfrac{|u(x)|^{2}}{|x-y|^{4m+2}}dx.
\]
\textbf{Case 1:} $(1-\theta)A\geq B$\textbf{.} In this case, we have
\begin{align*}
B  &  \geq C_{2}(N,m)\sup_{r>0,y\in\mathbb{R}^{N}}\int_{B(y,r)}\dfrac
{|u(x)|^{2}}{|x-y|^{4m+2}}dx\\
&  \geq C_{2}(N,m)\sup_{r>0,y\in\mathbb{R}^{N}}\dfrac{1}{r^{4m+2}}%
\int_{B(y,r)}|u(x)|^{2}dx.
\end{align*}
Moreover, since $A-B\geq\theta A$, it follows that
\begin{align*}
(A-B)^{\theta}B^{1-\theta}  &  \gtrsim\left\Vert \nabla\Delta^{m}u\right\Vert
_{L^{2}}^{2\theta}\left(  \sup_{r>0,y\in\mathbb{R}^{N}}\dfrac{1}{r^{4m+2}}%
\int_{B(y,r)}|u(x)|^{2}dx\right)  ^{1-\theta}\\
&  \gtrsim\left\Vert \nabla\Delta^{m}u\right\Vert _{L^{2}}^{2\theta}\left\Vert
u\right\Vert _{\mathbf{M}^{2,N-4m-2}}^{2(1-\theta)}.
\end{align*}
Then, it suffices to prove that
\begin{equation}
\left\Vert \nabla\Delta^{m}u\right\Vert _{L^{2}}^{2\theta}\left\Vert
u\right\Vert _{\mathbf{M}^{2,N-4m-2}}^{2(1-\theta)}\gtrsim\left\Vert
u\right\Vert _{L^{2N/(N-4m-2)}}^{2}. \label{eq 2.14}%
\end{equation}
Applying Theorem \ref{Thm 1.2} with $k=0,r=\frac{2N}{N-4m-2}$, and $m=2m+1$,
we have
\[
\frac{2N}{N-4m-2}=\frac{2(2m+1+s)}{s},
\]
which implies that $s=\frac{N-4m-2}{2}$, and
\begin{align*}
\left\Vert u\right\Vert _{L^{2N/(N-4m-2)}}  &  \lesssim\left\Vert u\right\Vert
_{\dot{B}^{-s}}^{(4m+2)/N}\left\Vert \nabla\Delta u\right\Vert _{L^{2}%
}^{(N-4m-2)/N}\\
&  \lesssim\left\Vert u\right\Vert _{M^{1,s}}^{(4m+2)/N}\left\Vert
\nabla\Delta^{m}u\right\Vert _{L^{2}}^{(N-4m-2)/N}\\
&  \lesssim\left\Vert u\right\Vert _{M^{2,2s}}^{(4m+2)/N}\left\Vert
\nabla\Delta^{m}u\right\Vert _{L^{2}}^{(N-4m-2)/N},
\end{align*}
where the second and third estimate are from Proposition \ref{lem 2.2}. The
last one is exactly \eqref{eq 2.14} with $\theta=\dfrac{N-4m-2}{N}$. Combining
this with the Sobolev inequality $\left\Vert u\right\Vert _{L^{2N/(N-4m-2)}%
}\lesssim\left\Vert \nabla\Delta^{m}u\right\Vert _{L^{2}}$, we get
\eqref{eq 2.14} with $\theta\geq\dfrac{N-4m-2}{N}.$\newline\textbf{Case 2:
}$(1-\theta)A<B$\textbf{.} Using the same notations for $\Omega$, $f$, and $v$
as in the proof of Theorem \ref{Tg1}, we have by the Hardy-Littlewood
inequality that
\[
\sup_{y\in\mathbb{R}^{N}}\int_{\mathbb{R}^{N}}\dfrac{|u(x)|^{2}}{|x-y|^{4m+2}%
}dx\leq\int_{\mathbb{R}^{N}}\dfrac{|u^{\ast}|^{2}}{|x|^{4m+2}}dx\leq
\int_{\mathbb{R}^{N}}\dfrac{|v|^{2}}{|x|^{4m+2}}dx.
\]
Also, by the P\'{o}lya-Szeg\"{o} inequality, we get
\[
A=\int_{\mathbb{R}^{N}}|\nabla\Delta^{m}u|^{2}dx=\int_{\mathbb{R}^{N}}|\nabla
f|^{2}dx\geq\int_{\mathbb{R}^{N}}|\nabla f^{\ast}|^{2}dx=\int_{\mathbb{R}^{N}%
}|\nabla\Delta^{m}v|^{2}dx.
\]
Therefore, by Theorem \ref{Thm 1.3}, we obtain
\[
(1-\theta)A<B\leq C_{2}(N,m)\int_{\mathbb{R}^{N}}\dfrac{|v|^{2}}{|x|^{4m+2}%
}dx\leq\int_{\mathbb{R}^{N}}|\nabla\Delta^{m}v|^{2}dx\leq A.
\]
Since the function $(A-x)^{\theta}x^{1-\theta}$ is non-increasing on $\left[
(1-\theta)A,A\right]  $, we get
\begin{align*}
&  \left(  \int_{\mathbb{R}^{N}}|\nabla\Delta^{m}u|^{2}dx-C_{2}(N,m)\sup
_{y\in\mathbb{R}^{N}}\int_{\mathbb{R}^{N}}\dfrac{|u(x)|^{2}}{|x-y|^{4m+2}%
}dx\right)  ^{\theta}\left(  \sup_{y\in\mathbb{R}^{N}}\int_{\mathbb{R}^{N}%
}\dfrac{|u(x)|^{2}}{|x-y|^{4m+2}}dx\right)  ^{1-\theta}\\
&  \gtrsim\left(  \int_{\mathbb{R}^{N}}|\nabla\Delta v|^{2}dx-C_{2}%
(N,m)\int_{\mathbb{R}^{N}}\dfrac{|v|^{2}}{|x|^{4m+2}}dx\right)  ^{\theta
}\left(  \int_{\mathbb{R}^{N}}\dfrac{|v|^{2}}{|x|^{4m+2}}dx\right)
^{1-\theta}.
\end{align*}
Also, note that $\left\Vert u\right\Vert _{L^{\frac{2N}{N-4m-2}}}%
^{2}=\left\Vert u^{\ast}\right\Vert _{L^{\frac{2N}{N-4m-2}}}^{2}\leq\left\Vert
v\right\Vert _{L^{\frac{2N}{N-4m-2}}}^{2}$. Therefore, it is now enough to
show that
\begin{equation}
\left(  \int_{\mathbb{R}^{N}}|\nabla\Delta^{m}u|^{2}dx-C_{2}(N,m)\int
_{\mathbb{R}^{N}}\dfrac{|u|^{2}}{|x|^{4m+2}}dx\right)  ^{\theta}\left(
\int_{\mathbb{R}^{N}}\dfrac{|u|^{2}}{|x|^{4m+2}}dx\right)  ^{1-\theta}%
\gtrsim\left\Vert u\right\Vert _{L^{\frac{2N}{N-4m-2}}}^{2} \label{2.15}%
\end{equation}
for all radial functions $u$, which is just a consequence of Theorem \ref{T3}.
Also, the condition $\theta=1/N$ is necessary in this case.
\end{proof}

\section{Proof of Theorem \ref{T2}}

Inspired by \cite[Theorem $2$]{DN23}, we would like to extend the
Hardy-Sobolev interpolation inequalities we have obtained in the previous
section by using Lorentz norms. First, we will recall the following results
that will be used in our proof.

\begin{theorem}
[Theorem 1.3 in \cite{DuLL22}]Let $N\geq p>1$, $0<R\leq\infty$, $A\geq0$ and
$B$ be $C^{1}$-functions on $\left(  0,R\right)  $. If $\left(  r^{N-1}%
A,r^{N-1}B\right)  $ is a $p$-Bessel pair on $\left(  0,R\right)  $, that is,
the ODE $\left(  A\left(  r\right)  \left\vert y^{\prime}\right\vert
^{p-2}y^{\prime}\right)  ^{\prime}+B\left(  r\right)  \left\vert y\right\vert
^{p-2}y=0$ has a positive solution $\varphi$ on $\left(  0,R\right)  $, then
for all smooth function $u$, we have%
\[%
{\displaystyle\int\limits_{B_{R}}}
A\left(  \left\vert x\right\vert \right)  \left\vert \nabla u\right\vert
^{p}dx-%
{\displaystyle\int\limits_{B_{R}}}
B\left(  \left\vert x\right\vert \right)  \left\vert u\right\vert ^{p}dx=%
{\displaystyle\int\limits_{B_{R}}}
A\left(  \left\vert x\right\vert \right)  \mathcal{C}_{p}\left(  \nabla
u,\varphi\nabla\left(  \frac{u}{\varphi}\right)  \right)  dx.
\]
Here $\mathcal{C}_{p}\left(  a,b\right)  =\left\vert a\right\vert ^{p}+\left(
p-1\right)  \left\vert b\right\vert ^{p}-p\left\vert a-b\right\vert
^{p-2}\left(  a-b\right)  \cdot b$.
\end{theorem}

In particular, when $p\geq2$, it is easy to verify that $\mathcal{C}%
_{p}\left(  a,b\right)  \geq c_{p}\left\vert y\right\vert ^{p}$. See
\cite{CKLL23, DuLL22}, for instance. Therefore, we have

\begin{lemma}
\label{lem 3.2}Let $N\geq p\geq2$, $A\geq0$ and $B$ be $C^{1}$-functions on
$\left(  0,\infty\right)  $. Assume that $\left(  r^{N-1}A,r^{N-1}B\right)  $
is a $p$-Bessel pair on $\left(  0,\infty\right)  $ with a positive solution
$\varphi$. We have, for all $u\in C_{0}^{\infty}(\mathbb{R}^{N})$ that%
\[
\int_{\mathbb{R}^{N}}A\left(  \left\vert x\right\vert \right)  \left\vert
\nabla u\right\vert ^{p}dx-\int_{\mathbb{R}^{N}}B\left(  \left\vert
x\right\vert \right)  \left\vert u\right\vert ^{p}dx\geq c_{p}\int
_{\mathbb{R}^{N}}A\left(  \left\vert x\right\vert \right)  \varphi^{p}|\nabla
v|^{p}dx,
\]
with $v=\frac{u}{\varphi}$.
\end{lemma}

We are now ready to present a proof of Theorem \ref{T2}.

\begin{proof}
[Proof of Theorem \ref{T2}]Let $A=\left\Vert \nabla u\right\Vert _{p,q}^{q}$
and
\[
B=\omega_{N}^{\frac{q}{p}-1}\left(  \dfrac{N-p}{p}\right)  ^{q}\sup
_{y\in\mathbb{R}^{N}}\int_{\mathbb{R}^{N}}\dfrac{|u(x)|^{q}}{|x-y|^{N+q-\frac
{Nq}{p}}}dx.
\]

\textbf{Case 1:} $(1-\theta)A\geq B$. It suffices to prove that
\[
\Vert\nabla u\Vert_{p,q}^{q\theta}\left(  \sup_{y\in\mathbb{R}^{N}}%
\int_{\mathbb{R}^{N}}\dfrac{|u(x)|^{q}}{|x-y|^{N+q-\frac{Nq}{p}}}dx\right)
^{1-\theta}\gtrsim\left\Vert u\right\Vert _{L^{p^{\ast},r}}^{q},
\]
or just need to prove
\[
\Vert\nabla u\Vert_{p,q}^{\theta}\left(  \sup_{r>0,y\in\mathbb{R}^{N}}%
\dfrac{1}{r^{N+q-\frac{Nq}{p}}}\int_{B(y,r)}|u(x)|^{q}dx\right)
^{(1-\theta)/q}\gtrsim\left\Vert u\right\Vert _{L^{p^{\ast},r}},
\]
which is equivalent to
\begin{equation}
\Vert\nabla u\Vert_{p,q}^{\theta}\Vert u\Vert_{M^{q,\frac{Nq}{p}-q}}%
^{1-\theta}\gtrsim\left\Vert u\right\Vert _{L^{p^{\ast},r}}. \label{eq 3.5}%
\end{equation}
By using \cite[Lemma $3.1$]{DLL22} with $\sigma=\frac{N-p}{p}$, $q=p^{\ast}$,
we get
\[
\left\vert u\left(  x\right)  \right\vert \lesssim\Vert u\Vert_{\dot
{B}^{-\sigma}}^{\frac{1}{1+\sigma}}[\mathcal{M}(\nabla u)(x)]^{\frac{\sigma
}{1+\sigma}}.
\]
Taking the $L^{p^{\ast},r}$-norm in both sides of the above inequality and
using the property of the maximal Hardy-Littlewood function, we have
\[
\Vert u\Vert_{L^{p^{\ast},r}}\lesssim\Vert u\Vert_{\dot{B}^{-\sigma}}%
^{\frac{1}{1+\sigma}}\Vert\mathcal{M}(\nabla u)\Vert_{L^{p,\frac{r\sigma
}{1+\sigma}}}^{\frac{\sigma}{1+\sigma}}\lesssim\Vert u\Vert_{\dot{B}^{-\sigma
}}^{\frac{1}{1+\sigma}}\Vert\nabla u\Vert_{L^{p,\frac{r\sigma}{1+\sigma}}%
}^{\frac{\sigma}{1+\sigma}}=\Vert u\Vert_{\dot{B}^{-\sigma}}^{\frac{p}{N}%
}\Vert\nabla u\Vert_{L^{p,\frac{r(N-p)}{N}}}^{\frac{N-p}{N}},
\]
provided that $\dfrac{r(N-p)}{N}\geq1$ or $r\geq\dfrac{p^{\ast}}{p}$. From
Proposition \ref{lem 2.2}, it follows that
\[
\Vert u\Vert_{L^{p^{\ast},r}}\lesssim\Vert u\Vert_{M^{q,\frac{q(N-p)}{p}}%
}^{\frac{p}{N}}\Vert\nabla u\Vert_{L^{p,\frac{r(N-p)}{N}}}^{\frac{N-p}{N}},
\]
which is exactly \eqref{eq 3.5} with $q=\dfrac{r(N-p)}{N}$ and $\theta
=\dfrac{N-p}{N}$. Combining this with the Sobolev inequality with Lorentz norm \cite{A77}
(with $p\leq r\leq p^{\ast}$) that
\[
\Vert\nabla u\Vert_{L^{p,\frac{r(N-p)}{N}}}\geq\Vert\nabla u\Vert_{L^{p}%
}\gtrsim\Vert u\Vert_{L^{p^{\ast},p}}\geq\Vert u\Vert_{L^{p^{\ast},r}},
\]
we deduce \eqref{eq 3.5} with $q=\dfrac{r(N-p)}{N}$ and $\theta\geq\dfrac
{N-p}{N}$.

\textbf{Case 2:} $(1-\theta)A<B$. In this case, the function $(A-B)^{\theta
}B^{1-\theta}$ is nonincreasing with respect to $B$ if $0\leq(1-\theta)A<B\leq
A$. Let $u^{\ast}$ be the standard Schwarz rearrangement of $u$. Then, by
Hardy-Littlewood inequality, we have
\[
\sup_{y\in\mathbb{R}^{N}}\int_{\mathbb{R}^{N}}\dfrac{|u(x)|^{q}}%
{|x-y|^{N+q-\frac{Nq}{p}}}dx\leq\int_{\mathbb{R}^{N}}\dfrac{|u^{\ast}(x)|^{q}%
}{|x|^{N+q-\frac{Nq}{p}}}dx.
\]
Moreover, from \cite{A77}, we also have the following P\'{o}lya-Szeg\"{o}
inequality
\[
\Vert\nabla u\Vert_{p,q}^{q}\geq\Vert\nabla u^{\ast}\Vert_{p,q}^{q},
\]
if $1\leq q\leq p$. Also,
\[
\Vert u\Vert_{L^{p^{\ast},r}}=\omega_{N}^{\frac{r-p^{\ast}}{p^{\ast}r}}\left(
\int_{\Omega^{\ast}}[u^{\ast}(x)|x|^{N/p^{\ast}}]^{r}\dfrac{dx}{|x|^{N}%
}\right)  ^{1/r}=\Vert u^{\ast}\Vert_{L^{p^{\ast},r}}.
\]
Therefore, it suffices to prove \eqref{eq 3.5} for all radial functions $u$,
i.e.
\begin{align}
&  \left(  \left\Vert \nabla u\right\Vert _{p,q}^{q}-\omega_{N}^{\frac{q}%
{p}-1}\left(  \dfrac{N-p}{p}\right)  ^{q}\int_{\mathbb{R}^{N}}\dfrac{|u|^{q}%
}{|x|^{N+q-\frac{Nq}{p}}}dx\right)  ^{\theta}\left(  \int_{\mathbb{R}^{N}%
}\dfrac{|u|^{q}}{|x|^{N+q-\frac{Nq}{p}}}dx\right)  ^{1-\theta}\nonumber\\
&  \gtrsim\left\Vert u\right\Vert _{L^{p^{\ast},r}}^{q}, \label{eq 3.6}%
\end{align}
where $q=\dfrac{r(N-p)}{N},\theta\geq\dfrac{N-p}{N}$, and $\frac{p^{\ast}}%
{p}\leq r\leq p^{\ast}$. We have, if $u$ is radial, then
\[
\Vert\nabla u\Vert_{p,q}^{q}=\omega_{N}^{\frac{q}{p}-1}\int_{\mathbb{R}^{N}%
}\dfrac{|\nabla u|^{q}}{|x|^{N(1-\frac{q}{p})}}dx,
\]
which follows that
\begin{align}
&  \left\Vert \nabla u\right\Vert _{p,q}^{q}-\omega_{N}^{\frac{q}{p}-1}\left(
\dfrac{N-p}{p}\right)  ^{q}\int_{\mathbb{R}^{N}}\dfrac{|u|^{q}}{|x|^{N+q-\frac
{Nq}{p}}}dx\nonumber\label{eq 3.7}\\
&  =\omega_{N}^{\frac{q}{p}-1}\left[  \int_{\mathbb{R}^{N}}\dfrac{|\nabla
u|^{q}}{|x|^{N(1-\frac{q}{p})}}dx-\left(  \dfrac{N-p}{p}\right)  ^{q}%
\int_{\mathbb{R}^{N}}\dfrac{|u|^{q}}{|x|^{N+q-\frac{Nq}{p}}}dx\right]  .
\end{align}
Next, we will use Lemma \ref{lem 3.2} to simplify \eqref{eq 3.7}. More
clearly, choosing $V\left(  r\right)  =r^{N(\frac{q}{p}-1)}$, $\varphi
=r^{\alpha}$, where $\alpha$ will be determined later, we have
\[
\varphi^{\prime}=\alpha r^{\alpha-1}%
\]
which implies that%
\[
V|\varphi^{\prime}|^{q-2}\varphi^{\prime}=\alpha|\alpha|^{q-2}r^{\frac{Nq}%
{p}-N+\alpha q-q-\alpha+1}%
\]
and
\[
(r^{N-1}V|\varphi^{\prime}|^{q-2}\varphi^{\prime})^{\prime}=\alpha
|\alpha|^{q-2}(r^{\frac{Nq}{p}+\alpha q-q-\alpha})^{\prime}=\alpha
|\alpha|^{q-2}\left(  \dfrac{Nq}{p}+\alpha q-q-\alpha\right)  r^{\frac{Nq}%
{p}+\alpha q-q-\alpha-1}.
\]
Choosing $\alpha=-\dfrac{N-p}{p}$, then
\[
(r^{N-1}V|\varphi^{\prime}|^{q-2}\varphi^{\prime})^{\prime}=\left(
\dfrac{|N-p|}{p}\right)  ^{q}r^{\frac{N-p}{p}-1}=\left(  \dfrac{|N-p|}%
{p}\right)  ^{q}r^{N-1}r^{\frac{Nq}{p}-q-N}\varphi^{q-1}.
\]
That is $\left(  r^{N-1}r^{N(\frac{q}{p}-1)},r^{N-1}\left(  \dfrac{|N-p|}%
{p}\right)  ^{q}r^{\frac{Nq}{p}-q-N}\right)  $ is a $p$-Bessel pair on
$\left(  0,\infty\right)  $ with the positive solution $\varphi=r^{-\dfrac
{N-p}{p}}$. Hence, Lemma \ref{lem 3.2} implies
\begin{align*}
\int_{\mathbb{R}^{N}}\dfrac{|\nabla u|^{q}}{|x|^{N(1-\frac{q}{p})}}dx-\left(
\dfrac{N-p}{p}\right)  ^{q}\int_{\mathbb{R}^{N}}\dfrac{|u|^{q}}{|x|^{N+q-\frac
{Nq}{p}}}dx  &  \geq c_{q}\int_{\mathbb{R}^{N}}\dfrac{|\nabla v|^{q}%
}{|x|^{\frac{q(N-p)}{p}}}|x|^{N(\frac{q}{p}-1)}dx\\
&  \geq c_{q}\int_{\mathbb{R}^{N}}\dfrac{|\nabla v|^{q}}{|x|^{N-q}}dx\\
&  \gtrsim\int_{0}^{\infty}\left\vert v^{\prime}\right\vert ^{q}s^{q-1}ds,
\end{align*}
where $u(x)=|x|^{-\frac{N-p}{p}}v(|x|)$. Thus,
\[
\text{LHS}_{\eqref{eq 3.6}}\gtrsim\left(  \int_{0}^{\infty}\left\vert
v^{\prime}\right\vert ^{q}s^{q-1}ds\right)  ^{\theta}\left(  \int_{0}^{\infty
}\dfrac{\left\vert v\right\vert ^{q}}{s}\right)  ^{1-\theta}.
\]
By Holder's inequality, we have with $\frac{1}{q}+\frac{1}{q^{\prime}}=1$,
\[
\left(  \int_{0}^{\infty}\left\vert v^{\prime}\right\vert ^{q}s^{q-1}%
ds\right)  ^{1/q}\left(  \int_{0}^{\infty}\dfrac{\left\vert v\right\vert ^{q}%
}{s}ds\right)  ^{1/q^{\prime}}\geq\left\Vert v\right\Vert _{\infty}^{q}.
\]
Let $\beta>0$ be arbitrary, we have
\[
\left(  \int_{0}^{\infty}\left\vert v^{\prime}\right\vert ^{q}s^{q-1}%
ds\right)  ^{\beta/q}\left(  \int_{0}^{\infty}\dfrac{\left\vert v\right\vert
^{q}}{s}ds\right)  ^{1+\beta/q^{\prime}}\geq\left\Vert v\right\Vert _{\infty
}^{\beta q}\int_{0}^{\infty}\dfrac{\left\vert v\right\vert ^{q}}{s}ds\geq
\int_{0}^{\infty}\dfrac{\left\vert v\right\vert ^{\beta q+q}}{s}ds.
\]
In \cite{DN23}, we have the following estimate for the Lorentz norm of a
function, i.e.
\[
\Vert u\Vert_{p^{\ast},s}^{s}\sim\int_{0}^{\infty}\dfrac{|f(r)|^{s}}{r}dr,
\]
with $u(x)=|x|^{-\frac{N-p}{p}}f(|x|)$. Thus,
\[
\left(  \int_{0}^{\infty}\left\vert v^{\prime}\right\vert ^{q}s^{q-1}%
ds\right)  ^{\beta/q}\left(  \int_{0}^{\infty}\dfrac{\left\vert v\right\vert
^{q}}{s}ds\right)  ^{1+\beta/q^{\prime}}\gtrsim\Vert u\Vert_{p^{\ast},m}^{m},
\]
where $m=\beta q+q$, which implies that
\[
\left(  \int_{0}^{\infty}\left\vert v^{\prime}\right\vert ^{q}s^{q-1}%
ds\right)  ^{\frac{\beta}{(1+\beta)q}}\left(  \int_{0}^{\infty}\dfrac
{\left\vert v\right\vert ^{q}}{s}ds\right)  ^{\frac{\beta+q^{\prime}}%
{(1+\beta)q^{\prime}}}\gtrsim\Vert u\Vert_{p^{\ast},m}^{q}\gtrsim\Vert
u\Vert_{p^{\ast},r}^{q},
\]
provided that $m\leq r$. Therefore, the proof will be completed if we can
determine $\beta>0$ such that $(\beta+1)q\leq r$ and $\dfrac{\beta}%
{(1+\beta)q}=\theta\geq\dfrac{N-p}{N}$. Therefore, we have $\beta\left(
N-qN+pq\right)  \geq(N-p)q$, which implies the following conditions:%
\[
N-qN+pq\geq0\Leftrightarrow q\leq\frac{N}{N-p}\Leftrightarrow r\leq\left(
\dfrac{p^{\ast}}{p}\right)  ^{2},
\]%
\[
\dfrac{Nq}{N-qN+pq}\leq r\Leftrightarrow\dfrac{(N-p)^{2}r-Np}{(N-p)^{2}%
r-N^{2}}\geq0\Leftrightarrow r\leq\dfrac{Np}{(N-p)^{2}}.
\]
Hence $r\leq\dfrac{Np}{(N-p)^{2}}$. Combining with the condition
$\frac{p^{\ast}}{p}\leq r\leq p^{\ast}$ and $p\leq r\leq p^{\ast}$, we have
\[
\max\left\{  \dfrac{N}{N-p},p\right\}  \leq r\leq\min\left\{  \dfrac
{Np}{(N-p)^{2}},\dfrac{Np}{N-p}\right\}  .
\]
Also, we have $\dfrac{\beta}{1+\beta}=q\theta$. That is $\beta=\dfrac{q\theta
}{1-q\theta}$. Therefore we also need $\left(  \dfrac{q\theta}{1-q\theta
}+1\right)  q\leq r$. That is $\dfrac{1}{1-q\theta}\leq\dfrac{r}{q}=\dfrac
{N}{N-p}$. Hence, we deduce that $\dfrac{N-p}{N}\leq\theta\leq\dfrac{p}{Nq}$.
\end{proof}

\end{document}